\documentclass[10pt]{amsart}

\usepackage{amsmath, amsfonts, amsthm, amssymb}
\usepackage[foot]{amsaddr}
\usepackage[margin=1.5in]{geometry}
\usepackage{bbm,amsmath, amsfonts, amsthm, amssymb}
\usepackage{mathtools}
\mathtoolsset{showonlyrefs}
\usepackage{mathrsfs}
\usepackage{hyperref}
\usepackage{xcolor}
\usepackage{graphicx}
\usepackage{braket}
\usepackage{tikz-cd}
\usepackage{enumerate}
\usepackage[ruled, vlined]{algorithm2e}
\usepackage{listings}
\usepackage{xcolor}
\usepackage{tikz}
\usetikzlibrary{quantikz}
\usepackage[textsize=tiny]{todonotes}
\usepackage{bbm}

\graphicspath{{Figures/}}

\definecolor{codegreen}{rgb}{0,0.6,0}
\definecolor{codegray}{rgb}{0.5,0.5,0.5}
\definecolor{codepurple}{rgb}{0.58,0,0.82}
\definecolor{backcolour}{rgb}{0.95,0.95,0.92}

\lstdefinestyle{mystyle}{
    backgroundcolor=\color{backcolour},   
    commentstyle=\color{codegreen},
    keywordstyle=\color{magenta},
    numberstyle=\tiny\color{codegray},
    stringstyle=\color{codepurple},
    basicstyle=\ttfamily\footnotesize,
    breakatwhitespace=false,         
    breaklines=true,                 
    captionpos=b,                    
    keepspaces=true,                 
    numbers=left,                    
    numbersep=5pt,                  
    showspaces=false,                
    showstringspaces=false,
    showtabs=false,                  
    tabsize=2
}
\numberwithin{equation}{section}
\newtheorem{theorem}{Theorem}[section]

\newtheorem{lemma}[theorem]{Lemma}
\newtheorem{proposition}[theorem]{Proposition}

\theoremstyle{definition}
\newtheorem{definition}[theorem]{Definition}
\newtheorem{remark}[theorem]{Remark}

\newtheorem{example}[theorem]{Example}

\newtheorem{assumption}[theorem]{Assumption}

\newcommand{\ind}{1\hspace{-2.1mm}{1}}
\newcommand{\eps}{\varepsilon}

\newcommand{\Cc}{\mathcal{C}}

\newcommand{\Ee}{\mathcal{E}}

\newcommand{\Hh}{\mathcal{H}}
\newcommand{\Hht}{\overline{\mathcal{H}}}

\newcommand{\Mm}{\mathcal{M}}

\newcommand{\Oo}{\mathcal{O}}

\newcommand{\Ss}{\mathcal{S}}

\newcommand{\Xx}{\mathcal{X}}
\newcommand{\Ig}{\mathtt{I}}
\newcommand{\Cg}{\mathtt{C}}
\newcommand{\Kg}{\mathtt{K}}
\newcommand{\Sg}{\mathtt{S}}
\newcommand{\Ag}{\mathtt{A}}
\newcommand{\Ug}{\mathtt{U}}

\newcommand{\Xg}{\mathtt{X}}
\newcommand{\Yg}{\mathtt{Y}}
\newcommand{\Zg}{\mathtt{Z}}
\newcommand{\Hg}{\mathtt{H}}
\newcommand{\Lg}{\mathtt{L}}

\newcommand{\Pg}{\mathtt{P}}
\newcommand{\Vg}{\mathtt{V}}

\newcommand{\half}{\frac{1}{2}}
\newcommand{\rhopm}{\rho}

\newcommand{\CC}{\mathbb{C}}
\newcommand{\DD}{\mathbb{D}}
\newcommand{\NN}{\mathbb{N}}
\newcommand{\RR}{\mathbb{R}}
\newcommand{\PP}{\mathbb{P}}
\newcommand{\VV}{\mathbb{V}}
\newcommand{\ZZ}{\mathbb{Z}}

\newcommand{\D}{\mathrm{d}}
\newcommand{\E}{\mathrm{e}}
\newcommand{\I}{\mathrm{i}}
\newcommand{\Tr}{\mathrm{Tr}}

\newcommand{\mf}{\mathfrak{m}}

\newcommand{\df}{\mathfrak{d}}
\newcommand{\pf}{\mathfrak{p}}
\newcommand{\qf}{\mathfrak{q}}

\newcommand{\Ff}{\mathcal{F}}

\newcommand{\omb}{\boldsymbol{\omega}}

\let\oldtocsection=\tocsection
\let\oldtocsubsection=\tocsubsection
\let\oldtocsubsubsection=\tocsubsubsection

\renewcommand{\tocsection}[2]{\hspace{0em}\oldtocsection{#1}{#2}}
\renewcommand{\tocsubsection}[2]{\hspace{2em}\oldtocsubsection{#1}{#2}}
\renewcommand{\tocsubsubsection}[2]{\hspace{4em}\oldtocsubsubsection{#1}{#2}}

\newcommand{\EE}{\mathbb{E}}

\newcommand{\oon}{o\left(\frac{1}{n}\right)}
\newcommand{\oonsq}{o\left(\frac{1}{\sqrt{n}}\right)}

\newcommand{\rhob}{\overline{\rho}}

\newcommand{\rhokz}{\rho_{k}^{0}}
\newcommand{\rhoko}{\rho_{k}^{1}}
\newcommand{\rhoieps}{\rho^{i,\eps}}
\newcommand{\rhokzeps}{\rho^{0,\eps}_{k}}

\begin{document}
\title{Stochastic limits of Quantum repeated measurements}
\author{Adeline Viot}
\address{Department of Mathematics, ENS Paris-Saclay}
\email{adeline.viot@ens-paris-saclay.fr}
\author{Antoine Jacquier}
\address{Department of Mathematics, Imperial College London}
\email{a.jacquier@imperial.ac.uk}
\author{Kostas Kardaras}
\address{Department of Statistics, London School of Economics}
\email{k.kardaras@lse.ac.uk}

\date{\today}
\thanks{We would like to thank Francesco Petruccione for enlightening discussions on collision models.
AJ is supported by the EPSRC Grant EP/W032643/1.}
\keywords{Belavkin equation, quantum measurement, stochastic Schr\"odinger, weak convergence}
\subjclass[2010]{45D05, 60F05, 60F17, 60H10, 81P68}
\maketitle

\begin{abstract}
We investigate quantum systems perturbed by noise in the form of repeated interactions between the system and the environment.
As the number of interactions (aka time steps) tends to infinity, we show, following the works by Pellegrini~\cite{pellegrini2008existence}
that this system converges to the solution of a Volterra stochastic differential equation.
This development sets  interesting future research paths at the intersection of quantum algorithms, stochastic differential equations, weak convergence and large deviations.
\end{abstract}
\tableofcontents


\section{Introduction}
The stochastic modelling of open quantum systems under continuous measurement has led to the development of a rich mathematical framework involving stochastic differential equations. Among these, the so-called Belavkin equations also known as stochastic Schrödinger equations~\cite{attal2006from,pellegrini_poisson2007,pellegrini2008existence} describe the random evolution of the quantum state of an open system undergoing continuous measurement. These equations capture the intrinsic randomness induced by quantum measurement and give rise to what are called \emph{quantum trajectories}.
An intuitive way to obtain such equations is through the \emph{repeated interaction model} in which the environment is modelled as a sequence of identical quantum systems, 
each of which interacts successively with the  system under consideration. 
After each interaction, a measurement is performed on the corresponding environmental unit, and the outcome influences the state of the small system. This setup leads to a discrete stochastic evolution of the system’s state. Under suitable scaling assumptions, one can rigorously prove that this discrete-time model converges, in the limit of infinitesimal interaction times, to a continuous stochastic equation of Belavkin type. This convergence relies on techniques from stochastic analysis—such as weak convergence of stochastic integrals—and functional analysis~\cite{kurtz1991wong, whitt2007martingale, dembozeitouni1998LDP}.

Open quantum systems, 
as quantum systems perturbed by external noise, are of fundamental importance in the current context of NISQ (Noisy Intermediate State Quantum Computing) environment, where quantum hardware suffer from noise in the form of partial decoherence of qubits and imperfect quantum gates.
Understanding this noise is of fundamental importance.
However, discrete formulations of quantum systems, akin to discrete stochastic systems, are difficult to grasp, and we hope that an accurate formulation of their continuous limit will help in understand their properties.

We develop here extensions of this standard repeated interaction model in several directions, accounting for more general -- and seemingly more realistic -- versions of noise. 
We first consider discrete dynamics in which the interaction between system and environment alternates between two distinct interaction operators, derive the corresponding continuous limits equations and analyse how a perturbation in the interaction operator affects the resulting quantum trajectories.
We then study memory effects in quantum dynamics (where the evolution depends not only on the current state of the system, but also on its past states) and show that the continuous limit turns out to be a stochastic Volterra equation.

The paper is organised as follows:
Section~\ref{sec:Background} recalls the discrete repeated interaction model and its convergence to a continuous-time stochastic differential equation. It also serves as a reminder of the essential tools of quantum computing calculus and sets the notations for the rest of the analysis.
We develop our first extension in Section~\ref{sec:Perturb}, modifying the unitary coupling at every time step and perturbing it using the machinery of quantum channels, 
and prove the continuous-time limit of this extension.
We further show that the perturbation can be analysis through the lens of large deviations techniques, providing precise estimates, useful to control the error estimates.
Finally, in Section~\ref{sec:memory}, we introduce a version with memory and prove its convergence to a stochastic Volterra equation.

\section{Background on convergence of discrete quantum trajectories}\label{sec:Background}
We first provide a self-contained overview of classical convergence results of discrete quantum trajectories to continuous-time stochastic models
and introduce the main notations and tools for the subsequent analysis.
We shall consider the \emph{discrete repeated interaction model}, where a quantum system interacts sequentially with an infinite chain of identical environment units.
After each interaction, a measurement is performed on the environment, resulting in a discrete-time stochastic evolution of the system's state.
We then recall the diffusive Belavkin stochastic differential equation describing the evolution of a quantum system under continuous indirect measurement, and discuss the convergence of the discrete model towards its continuous version.
This setup is based on and borrowed from~\cite{pellegrini2008existence} and related papers~\cite{attal2006from,gough2004stochastic}.

\subsection{Notations and setup}
Given a (complex) Hilbert space~$\Hh$,
a quantum state is represented via its density operator, namely a positive semi-definite Hermitian matrix with unit trace acting on~$\Hh$, and we shall write~$\Ss(\Hh)$
the set of all such states.
Quantum measurements are performed via \emph{observables}, that is self-adjoint operators on~$\Hh$ whose (real) eigenvalues correspond to the outcomes of the measurement.
We refer the reader to~\cite[Chapter~2]{nielsen_chuang_2000} for general terminology and background on this formalism for quantum computing.

\subsection{The quantum repeated measurement model}\label{sec:rep_measurement_model}
We consider a quantum system with Hilbert space~$\Hh_0$ interacting sequentially with an infinite chain of identical and independent quantum systems modelling the environment.
Each element~$k$ of the environment is described by a Hilbert space $\Hh_k=\Hh$ and the full space is 
$\Hht := \Hh_0 \otimes \bigotimes_{k \geq 1} \Hh_k$.
Each interaction lasts for a time length~$\tau>0$ and is governed by the total Hamiltonian
$$
\Hg_{\text{tot}} = \Hg_{0} \otimes \Ig + \Ig \otimes \Hg + \Hg_{\text{int}},
$$
acting on~$\Hh_0\otimes\Hh$,
leading to the unitary evolution operator $\Ug = \E^{-\I \tau \Hg_{\text{tot}}}$.
The $k$-th interaction is implemented by the unitary operator~$\Ug_k$, acting as $\Ug$ on the bipartite system $\Hh_0 \otimes \Hh_k$ and as the identity elsewhere.
After each interaction, a measurement is performed on~$\Hh_k$ via a fixed observable~$\Ag$ with spectral decomposition $\Ag = \sum_{j=1}^{\mf} \lambda_j \Pg_j$, where $\{\Pg_j\}_{j=1,\ldots,\mf}$ are orthogonal projectors.
The probability of observing  the eigenvalue~$\lambda_j$ for the
quantum state $\sigma \in \Ss(\Hht)$
is $\Tr[\sigma \Pg_j^k]$ by Born's rule and the post-measurement state becomes
$\sigma_j = \frac{\Pg_j^k \sigma \Pg_j^k}{\Tr[\sigma \Pg_j^k]}$,
where we denote~$\Pg^k_j$ the operator on~$\Hht$ which acts as~$\Pg_j$ on~$\Hh_k$ and as the identity on~$\Hh_i$ for $i\ne k$.
We initialise the system as
$\sigma = \rho \otimes \bigotimes_{j \geq 1}\beta$,
with $\rho\in\Ss(\Hh_0)$ and $\beta\in\Ss(\Hh)$.
After~$k$ interactions, the state then becomes
$\sigma_k = \Vg_k \sigma \Vg_k^\dagger$,
where 
we define recursively 
$\Vg_{k+1} = \Ug_{k+1} \Vg_k$,
for $k\geq 0$,
starting from 
$\Vg_0 = \Ig$.

Measurement outcomes are modelled by a sequence $\omb = (\omega_k)_{k \geq 1} \in \Omega^{\NN^*} = \{1, \dots, \mf\}^{\NN^*}$ and the cylinder $\sigma$-algebra generated by
$\Lambda_{i_1, \dots, i_k} = \{ \omb \in \Omega^{\NN^*} \mid \omega_1 = i_1, \dots, \omega_k = i_k \}$.
Because $\Ug_j$ commutes with $\Pg^k$ for each $k < j$, the (non-normalised) post-measurement state after outcomes $(i_1, \dots, i_k)$ reads
$
\widetilde{\sigma}(i_1, \dots, i_k)
= \left(\Pg_{i_k}^k \dots \Pg_{i_1}^1\right) \sigma_k \left(\Pg_{i_1}^1 \dots \Pg_{i_k}^k\right)$,
with associated probability
$
\PP[\Lambda_{i_1, \dots, i_k}] = \Tr[\widetilde{\sigma}(i_1, \dots, i_k)],
$
which defines a probability measure on $\Omega^{\NN^*}$.
The resulting discrete quantum trajectory is the normalised sequence
\begin{equation}\label{eq:pellegSequence}
\sigma_k(\omb) := \frac{\widetilde{\sigma}(\omega_1, \dots, \omega_k)}{\Tr[\widetilde{\sigma}(\omega_1, \dots, \omega_k)]}.
\end{equation}
The following proposition summarises the key results from~\cite[Proposition~2.1, Theorem~2.2, Theorem~ 2.3]{pellegrini2008existence} and characterises this sequence:

\begin{theorem}
\label{thm:QuantumTraj_Disc}
The trajectory $(\sigma_k)_{k\geq 0}$ is an $\Hht$-valued Markov chain
and satisfies
\begin{equation}\label{eq : sequence rho_k with a singular U}
\sigma_{k+1}(\omb) = \frac{\Pg^{k+1}_{\omega_{k+1}} \Ug_{k+1}
\sigma_k(\omb)
\Ug_{k+1}^\dagger \Pg^{k+1}_{\omega_{k+1}}}{\Tr[\sigma_k(\omb) \Ug_{k+1}^\dagger \Pg^{k+1}_{\omega_{k+1}} \Ug_{k+1}]},
\qquad\text{for all }k\geq 0,
\omb \in \Omega^{\NN^\ast}.
\end{equation}
Conditionally on 
$\{\sigma_{k} = \theta_k\}$, 
$\displaystyle 
\PP\left[\sigma_{k+1} = \theta_{k+1}^{i} \mid \sigma_k = \theta_k\right]
 = \Tr\left[\Ug_{k+1}(\theta_k \otimes \beta) \Ug_{k+1}^\dagger \Pg^{k+1}_i\right]$
where 
$$
\left\{\theta_{k+1}^{i} = \frac{\Pg^{k+1}_i \Ug_{k+1}(\theta_k \otimes \beta) \Ug_{k+1}^\dagger \Pg^{k+1}_i}{\Tr[\Ug_{k+1}(\theta_k \otimes \beta) \Ug_{k+1}^\dagger \Pg^{k+1}_i]}\right\}_{1, \dots, \mf},
\qquad\text{for all } k\geq 0.
$$
\end{theorem}

The evolution of the couple (system, environment) is unitary, as required by the axioms of quantum mechanics. However, we are ultimately interested by the evolution of the system itself, which can be seen through \emph{partial trace} operations:
\begin{definition}[Partial trace]
Given two Hilbert spaces~$\Hh_1$ and~$\Hh_2$ and $\rho \in \Ss(\Hh_1 \otimes \Hh_2)$, the unique state~$\rho_1\in\Ss(\Hh_1)$ such that $\Tr_{\Hh_1}[\mathtt{D}\rho_1]
= \Tr_{\Hh_1 \otimes \Hh_2} [(\mathtt{D} \otimes \Ig_{\Hh_2}) \rho]$
for all bounded operators~$\mathtt{D}$ on~$\Hh_1$
is called the \emph{partial trace} of $\rho$ with respect to $\Hh_2$.
\end{definition}

Probabilistically, the partial trace corresponds to a conditional expectation
and we write~$\EE_0[\cdot]$ as the partial trace over~$\Hh_0$ with respect to $\bigotimes_{k \geq 1} \Hh_k$ of any state on~$\Hht$.  
For $\omb \in \Omega^{\NN^\ast}$, 
the discrete quantum trajectory defined by
\begin{equation} \label{eq:discrete_trajectory_single_U}
\rho_k(\omb) := \EE_0[\sigma_k(\omb)], \qquad\text{for }k\geq 0,
\end{equation}
on $\Hh_0$ represents the conditional evolution of the quantum system given the outcomes of the repeated indirect measurements.

We shall from now on consider $\Hh_0=\Hh_k=\CC^2$ for all $k\geq 1$ 
(and therefore set $\mf=2$)
and denote $\{e_{0}, e_{1}\}$ an orthonormal basis of~$\CC^2$. 
We choose $\beta= e_{0} e_{0}^\dagger$ and
use the  basis 
$\{
e_{0}\otimes e_{0},
e_{1}\otimes e_{0},
e_{0}\otimes e_{1},
e_{1}\otimes e_{1}\}$ 
for $\Hh_0 \otimes \Hh$.
Typically in Quantum Computing $e_{0} = \ket{0}=(1,0)^\top$ and $e_{1} = \ket{1}=(0,1)^\top$ denote the standard basis elements of~$\CC^2$ in Dirac's notations
and therefore $\beta=\ket{0}\bra{0}$.
In this basis, the unitary operator~$\Ug$ acting on $\Hh_0 \otimes \Hh$ can be written in block matrix form
\begin{equation}\label{eq:unitaryU}
\Ug = 
\begin{pmatrix}
\Ug_{00} & \Ug_{01} \\
\Ug_{10} & \Ug_{11}
\end{pmatrix},
\end{equation}
where each~$\Ug_{ij}$ is a $2\times 2$ matrix.
The choice $\beta = e_{0} e_{0}^\dagger$ 
ensures that only the blocks~$\Ug_{00}$ and~$\Ug_{01}$ influence the evolution of the discrete trajectory $(\rho_k)_{k\geq 0}$, and indeed we have
\begin{equation}\label{eq : og mu k+1}
\Ug (\rho_k\otimes \beta)\Ug^\dagger
=\sum_{i,j=0}^1 \Ug_{i0} \rho_{k} \Ug_{j0}^\dagger\otimes e_i e_j^\dagger
=\sum_{i,j=0}^1 \Ug_{i0} \rho_{k} \Ug_{j0}^\dagger\otimes 
\ket{i}\bra{j}.
\end{equation}
Thanks to Theorem~\ref{thm:QuantumTraj_Disc}, we define the two possible non-normalised states as 
\begin{equation} \label{eq:L0 and L1}
\rho_{k}^{j}  := \EE_0\Big[ (\Ig \otimes \Pg_{j}) \Ug (\rho_k\otimes \beta)\Ug^\dagger (\Ig \otimes \Pg_{j}) \Big], 
\qquad\text{for }j=0,1.
\end{equation}
The former appears with probability  
$p_{k+1} = \Tr[\rhokz]$ and the latter with probability  
$q_{k+1} =\Tr[\rhoko] = 1-p_{k+1}$.
For any $k\geq 1$, introduce the random variable 
$$
\nu_{k} :=
\begin{cases}
0, & \text{with probability } p_{k}, \\
1, & \text{with probability } q_{k}.
\end{cases}
$$
Using this notation and provided that both probabilities are not equal to zero, the discrete quantum trajectory is described by
\begin{equation}
\label{eq:discrete_evolution}
\rho_{k+1}  = \frac{\rhokz}{p_{k+1}} (1 - \nu_{k+1}) + \frac{\rhoko}{q_{k+1}} \nu_{k+1}
 = \rhokz + \rhoko - \left(\sqrt{\frac{q_{k+1}}{p_{k+1}}} \rhokz - \sqrt{\frac{p_{k+1}}{q_{k+1}}} \rhoko \right) X_{k+1},
\end{equation}
where we introduced the random variable
\begin{equation}\label{eq:def_X_k}
X_{k} := \frac{\nu_{k} - q_{k}}{\sqrt{q_{k} p_{k}}},
\qquad\text{for any }k\geq 1,
\end{equation}
with associated filtration 
$
\Ff_k := \sigma(X_i,\; i \leq k)$ on $\{0, 1\}^{\NN^\ast}$,
so that $(X_k)_{k\geq 1}$ forms a sequence satisfying
$\EE[X_{k+1} \mid \Ff_k] = 0$
and
$\EE[X_{k+1}^2 \mid \Ff_k] = 1$ for all $k\geq 0$.
\begin{remark}\label{rem:Proba01}
In the case where $p_{k+1} = 0$, then $\nu_{k+1} = 1$ almost surely
and~\eqref{eq:discrete_evolution} needs to be amended as
$\rho_{k+1} = \rho_{k}^{1}$ almost surely.
If $p_{k+1} = 1$, then $\nu_{k+1} = 0$ and $X_{k+1} = $ almost surely
and hence
$\rho_{k+1} = \rho_{k}^{0}$ almost surely.
We will discard these (trivial) limiting cases in the rest of the analysis.
\end{remark}

\subsection{Belavkin equations}

We now consider a two-level system in $\Hh_0=\CC^2$ interacting with an environment, whose global evolution is governed by a unitary process $(\Ug_t)_{t\geq 0}$ solving a quantum stochastic differential equation. Without measurement, the reduced evolution of the system is described by a semigroup with Lindblad generator~$\Lg$, satisfying the \emph{Master equation} (also called Gorini-Kossakowski-Sudarshan-Lindblad equation)~\cite[Chapter~8.4.1]{nielsen_chuang_2000}:
\begin{equation} \label{eq:master_equation}
\frac{\D\rho_t}{\D t} = \Lg(\rho_t) = -\I[\Hg_{0}, \rho_t] - \frac{1}{2} 
\{\Cg \Cg^\dagger, \rho_t \} + \Cg \rho_t \Cg^\dagger,
\end{equation}
starting from $\rho_0\in\Ss(\Hh_0)$ with $\Hg_{0}$ the system Hamiltonian and~$\Cg$ a given operator on~$\Hh_0$.
Recall that $[\cdot,\cdot]$ denotes the Lie bracket (matrix commutator) $[A,B] := AB - BA$ and $\{\cdot,\cdot\}$ the anticommutator
$\{A,B\} := AB + BA$.
We consider a noisy version of~\eqref{eq:master_equation} called 
the \emph{diffusive Belavkin equation}~\cite{belavkin1994nondemolition} that accounts for continuous monitoring:
\begin{equation}\label{eq:Belavkin_equation_og_form}
\D \rho_t = \Lg(\rho_t) \D t
+ \Big( \rho_t \Cg^\dagger + \Cg \rho_t - \Tr[\rho_t (\Cg + \Cg^\dagger)] \rho_t \Big) \D W_t, 
\qquad \rho_0\in\Ss(\Hh_0),
\end{equation}
for some Brownian motion~$W$
supported on a filtered probability space $(\Omega, \Ff, \{\Ff_t\}_{t \geq 0}, \PP)$.

\begin{remark}
A jump-version of the Belavkin equation, the \emph{Poisson Belavkin equation}, also exists, to model jump-type measurements, and reads~\cite{pellegrini_poisson2007}
\begin{equation}\label{poisson SDE}
\D\rho_t = \Lg(\rho_t)\D t + \left( \frac{\Cg\rho_t \Cg^\dagger}{\Tr[\Cg\rho_t \Cg^\dagger]} - \rho_t \right) 
\left(\D \widetilde{N}_t - \Tr[\Cg\rho_t \Cg^\dagger]\D t\right),
\end{equation}
where $\widetilde{N}$ is a counting process.
We shall not use it here though.
\end{remark}
Existence and well-posedness of~\eqref{eq:Belavkin_equation_og_form}
is guaranteed by the following ~\cite[Theorem~3.2]{pellegrini2008existence}:

\begin{theorem}\label{thm:SDE}
The SDE~\eqref{eq:Belavkin_equation_og_form} 
admits a unique strong solution in $\Ss(\Hh_0)$.
\end{theorem}

\subsection{Convergence of the discrete model to the solution of the Belavkin equation}

We now discretise the interval $[0,1]$ into~$n$ sub-intervals of size $\frac{1}{n}$, so that the unitary~$\Ug = \Ug(n)$ reads, in block form,
\begin{equation}\label{eq:unitaryU2}
\Ug(n) =
\begin{pmatrix}
\Ug_{00}(n) & \Ug_{01}(n) \\
\Ug_{10}(n) & \Ug_{11}(n)
\end{pmatrix},
\end{equation}
where each $\Ug_{ij}(n)$ acts on the system Hilbert space $\Hh_0$.
The following links the asymptotic expansions of the blocks~$\Ug_{ij}(n)$ with the corresponding generator of the unitary $\Ug(n)$:

\begin{proposition}\label{prop:corresponding_Hamiltonian}
Assume that, for $i,j=0,1$,
as $n$ tends to infinity,
$$
\Ug_{ij}(n) = \Ug_{ij}^{(0)} + \frac{1}{\sqrt{n}}\Ug_{ij}^{(1)} + \frac{1}{n}\Ug_{ij}^{(2)} 
+ o\left(\frac{1}{n}\right).
$$
Then there exists a Hermitian matrix~$\Hg(n)$ such that $\Ug(n) = \exp\big\{ \frac{\I}{n} \Hg(n) \big\}$ with
$\Hg(n) = nD +\sqrt{n}E +F + o(1)$,
as $n$ tends to infinity, for some matrices $D$, $E$, $F$.
\end{proposition}
\begin{proof}
Consider the ansatz $\Ug(n) = \Ug^{(0)}+\frac{1}{\sqrt{n}}\Ug^{(1)}+\frac{1}{n}\Ug^{(2)}+\oon$. 
Since~$\Ug(n)$ is unitary,
identifying the terms of order~1 shows that~$\Ug^{(0)}$ is also unitary. 
Assume that $H(n) = nD +\sqrt{n}E +F +o(1)$ for some matrices $D,E,F$.
By Lemma~\ref{lemma : matrix exponential} 
(with~$\phi$ defined therein),
$$
\exp\left\{\frac{\I}{n}H(n)\right\}
= \E^{\I D} \left[\Ig+\frac{\I}{\sqrt{n}}\phi_{\I D}(E) + \frac{\I}{n}\left(\phi_{\I D}(F)-\half \phi_{\I D}^{(2)}(E)\right)\right] + \oon.
$$
Hence
$\Ug^{(0)}=\E^{\I D}$,
$\Ug^{(1)}=\I \E^{\I D} \phi_{\I D}(E)$
and 
$\Ug^{(2)}=\E^{\I D}
\left(\I \phi_{\I D}(F)-\half \phi_{\I D}^{(2)}(E)\right)$.
Since~$\Ug^{(0)}$ is unitary,
it admits the spectral decomposition
$\Ug^{(0)} = \Vg\,\text{Diag}(\E^{\I \theta_1},\dots, \E^{\I \theta_4})\,\Vg^\dagger$ with~$\Vg$ unitary,
and~$D$ is Hermitian. 
We choose $\theta_j$ so that if $k< l$ exists with $\theta_k-\theta_l \in 2\pi \ZZ\setminus\{0\}$, $\theta_l$ is reset to $\theta_k$.
and hence $\phi_{\I D}$ is invertible via Lemma~\ref{lem:invertphi_{iD}}.
Then
$$
E=-\I\phi_{\I D}^{-1}\left(\Ug^{(0)\dagger} \Ug^{(1)}\right)
\qquad\text{and}\qquad
F=-\I \phi^{-1}_{\I D}
\left[\Ug^{(0)\dagger}\Ug^{(2)}+\half\phi_{\I D}^{(2)}\left(\phi_{\I D}^{-1}(\Ug^{(0)\dagger} \Ug^{(1)})\right)\right],
$$ 
and the proposition follows.
\end{proof}

Returning to the quantum trajectory $(\rho_k(n))_k$, we consider the asymptotic expansions
\begin{equation}\label{eq:L00}
\Ug_{00}(n) = \Ig - \frac{1}{n} \left(\I \Hg_{0} + \frac{\Cg \Cg^\dagger}{2} \right) + o\left( \frac{1}{n} \right)
\qquad\text{and}\qquad
\Ug_{10}(n) = \frac{\Cg}{\sqrt{n}} + o\left( \frac{1}{n} \right),
\end{equation}
as~$n$ tends to infinity.
Note that only~$\Ug_{00}$ and~$\Ug_{10}$ need to be approximated by~\eqref{eq : og mu k+1}, since the dynamics of~$(\rho_k)$ depend only on them.
From now on, we write $\rho_k$ for $\rho_k(n)$ when~$n$ is fixed.
The following assumption will be key throughout the whole paper:

\begin{assumption}\label{assu:NonDiag}
The observable $\Ag=\lambda_0\Pg_0 + \lambda_1 \Pg_1$ is non-diagonal in the basis $\{e_0,e_1\}$.
\end{assumption}

By \emph{non-diagonal}, we mean that at least one spectral projector of~$\Ag$ (either~$\Pg_0$ or~$\Pg_1$) is not diagonal in the basis $\{e_0,e_1\}$, equivalently
$\{e_0,e_1\}$ is not an eigenbasis of~$\Ag$. This is also equivalent to $\langle e_i,\Pg_j e_i\rangle \notin \{i,1\}$ for all $i,j$. 
Here, since there are only two projectors, assuming \emph{for either} is equivalent to assuming \emph{for all}.
Later, in higher dimensions, \emph{for all} will be the correct assumption.

If both $\Pg_j$ are not diagonal in the basis $(e_0,e_1)$, since their eigenvalues are $0$ and $1$, this implies that $\langle e_0,\Pg_j e_0\rangle \notin \{0,1\}$ for all $j$. Conversely, let us assume that $\langle e_0,\Pg_j e_0\rangle \notin \{0,1\}$. Note that $\Tr[\Pg_j]=1$ because the $\Pg_j$ are not equal to $0$ or $\Ig$ and the dimension of $\Hh$ is $2$. Then $\langle e_1,\Pg_j e_1\rangle \notin \{0,1\}$. Hence the $\Pg_j$s are not diagonal in the basis.
The following proposition, proved in Appendix~\ref{proof:prop__increment_of_rho_k}, provides a small-step expansion for the evolution of~$(\rho_k)$.

\begin{proposition}\label{prop:increment_of_rho_k}
Under Assumption~\ref{assu:NonDiag}, 
the sequence~\eqref{eq:discrete_evolution} satisfies
\begin{equation}\label{increment of rho_k}
\rho_{k+1} - \rho_{k} = \frac{1}{n} \Lg(\rho_{k})
+ \Big( \Cg_{\gamma} \rho_{k} + \rho_{k} \Cg_{\gamma}^\dagger - \Tr\left[\rho_{k}(\Cg_{\gamma} + \Cg_{\gamma}^\dagger)\right]
\rho_{k} + o(1) \Big) \frac{X_{k+1}}{\sqrt{n}} + \oon,
\end{equation}
where $\Cg_{\gamma}=\gamma \Cg$ with $\gamma\in\CC\setminus\{0\}$ depending on~$\Ag$ and on the basis $\{e_0,e_1\}$, and
$$
\Lg(\rho_k) := 
\Cg \rho_k \Cg^\dagger -\I[\Hg_{0}, \rho_k] - \frac{1}{2} 
\{\Cg \Cg^\dagger, \rho_k \}.
$$
\end{proposition}

\begin{remark}\label{rem:gamma1}
The precise value of~$\gamma$ is of no particular importance for the rest (as long as it is non zero), and we shall hence take $\gamma=1$.
\end{remark}
\begin{lemma}\label{lemma : A diagonal poisson case}
If $\Ag=\lambda_0\begin{pmatrix}
1&0\\0&0
\end{pmatrix}+\lambda_1\begin{pmatrix}
0&0\\0&1
\end{pmatrix}$ is diagonal, then 
$$
\rho_{k+1} - \rho_{k} 
 = \frac{1}{n}\Lg(\rho_{k})
+ \left[ \frac{\Cg\rho_{k} \Cg^\dagger}{\Tr[\Cg\rho_{k}\Cg^\dagger]}-\rho_{k} +o(1)\right]\sqrt{p_{k+1}q_{k+1}}X_{k+1}
+\oon.
$$
\end{lemma}

\begin{proof}    With~\eqref{eq:discrete_evolution}, note that the only term for which the proof of Proposition~\ref{prop:increment_of_rho_k} must be adapted is $-\sqrt{\frac{q_{k+1}}{p_{k+1}}} \rhokz + \sqrt{\frac{p_{k+1}}{q_{k+1}}} \rhoko$. By factoring by $\sqrt{p_{k+1}q_{k+1}}$, let us calculate $-\frac{\rhokz}{p_{k+1}}+\frac{\rhoko}{q_{k+1}}$.
With the proof of Proposition~\ref{prop:increment_of_rho_k}, recall that 
$$
\rhokz=\displaystyle \rho_{k}-\frac{\I}{n}[\Hg_{0},\rho_{k}]-\frac{1}{2n}\{\Cg \Cg^\dagger,\rho_{k}\} 
 + \oon
 \qquad\text{and}\qquad \rhoko=\displaystyle \frac{1}{n}\Cg\rho_{k}\Cg^\dagger + \oon.
 $$
Then 
$p_{k+1}=1-\frac{1}{n}\Tr[\Cg\rho_{k}\Cg^\dagger]+\oon$
and $q_{k+1}=1-p_{k+1} = \Tr[\Cg\rho_{k}\Cg^\dagger]+\oon$,
hence 
$$
\frac{\rhokz}{p_{k+1}}=\rho_{k}+o(1)
 \quad\text{and}\quad \frac{\rhoko}{q_{k+1}} = \Cg\rho_{k}\Cg^\dagger + o(1),
$$
and the proposition follows.
\end{proof}

In the non-diagonal case in Proposition~\ref{prop:increment_of_rho_k},
two types of convergence were established in~\cite{pellegrini2008existence}: 
in expectation~\cite[Theorem~4.1]{pellegrini2008existence} 
and in distribution~\cite[Theorem~4.4]{pellegrini2008existence}. 
The following two theorems summarise the results of the  previous subsections. 
Note that the convergence in distribution is also established if~$\Ag$ is diagonal (as in Lemma~\ref{lemma : A diagonal poisson case}) in \cite[Theorem~4]{pellegrini_poisson2007}.

\begin{theorem}\label{thm:Convergence_in_E}
Under Assumption~\ref{assu:NonDiag}, 
the map
$t \mapsto \EE[\rho_{\lfloor nt \rfloor}(n)]$  
converges in $L^\infty([0, 1])$ as~$n$ tends to infinity to
the unique solution to
$\frac{\D\nu_t}{\D t} = \Lg(\nu_t)$,
starting from $\nu_0=\rho_0$.
\end{theorem}

Recall~\cite[Chapter~3]{billingsley2013convergence} that the Skorokhod space $\DD([0,1],\RR^d)$ is the set of $\RR^d$-valued c\`adl\`ag processes
and is particularly suited for quantum trajectories as their evolution typically features small jumps induced by discrete measurement events.
For $f,g\in \DD([0,1],\RR^d)$, 
let 
\begin{equation*}
\df(f,g) :=
\inf\left\{\eps>0: 
\begin{array}{l}
\sup_{0\leq t \leq 1} |t-F(t)|\leq \eps,\\
\sup_{0\leq t \leq 1} \lVert f(t)-g(F(t))\rVert \leq \eps, \end{array}
\text{ for some }F\in\Lambda
\right\},
\end{equation*}
where $\Lambda$ is the space of  continuous and increasing functions from $[0,1]$ to $[0,1]$.
A sequence $(x_n)_{n\in\NN}\in\Cc([0,1],\RR^d)^{\NN}$ is said to converge to  $x\in\Cc([0,1],\RR^d)$ in $(\DD([0,1],\RR^d),\df)$ if and only if it converges to~$x$ in $(\Cc([0,1],\RR^d),\|\cdot\|_\infty)$.
We denote $X_n \Rightarrow X$ for the weak convergence of a sequence
$(X_n)$, either in 
$\DD([0,1],\RR^d)$ or in $\Cc([0,1],\RR^d)$,
and by $X_n(t) \Rightarrow X(t)$ the convergence in law at some fixed time $t\in [0,1]$.
Obviously weak convergence in $\DD([0,1],\RR^d)$ implies 
convergence in law in~$\RR^d$ for all $t\in [0,1]$.

\begin{theorem}
Under Assumption~\ref{assu:NonDiag}, 
for $\rho_0\in\Ss(\Hh_0)$,
the discrete quantum trajectory $(\rho_{\lfloor n\cdot \rfloor}(n))_n$ converges in distribution to the solution to~\eqref{eq:Belavkin_equation_og_form} as~$n$ to infinity.
\end{theorem}

\section{Perturbing the unitary coupling}\label{sec:Perturb}
Until now, we have studied the behaviour of a sequence of quantum trajectories arising from repeated interactions with an environment (governed by a fixed unitary operator~$\Ug$) and measurements associated with a fixed observable~$\Ag$.
This model yields convergence, in a suitable scaling regime, towards a solution of the diffusive Belavkin equation~\eqref{eq:master_equation}.
We now aim to understand how changes in the interaction mechanism (specifically in the choice of unitary operators) affect the limiting stochastic evolution of the system.
We first consider a modified model in which the $k$-th interaction is not governed by the same unitary for all $k$, but alternates between two unitaries $\Ug^+$ and $\Ug^-$ depending on the parity of $k$. We analyse how this periodic structure in the dynamics influences the resulting continuous-time limit.
We then turn to a stability question,  comparing the limiting trajectories associated with two sequences of repeated interactions governed by the same observable~$\Ag$, but with two slightly different interaction unitaries. One of them represents an ideal unitary evolution, while the other corresponds to a perturbed version. Our goal is to quantify the deviation between the two limiting dynamics and to assess the sensitivity of the model to imperfections in the unitary coupling,
as is the case in current NISQ (Noisy Intermediate State Quantum Computing) hardware.

\subsection{Alternating unitary operators}\label{sec:section2a}

We again consider $\Hh_0=\Hh=\CC^2$, 
a Hamiltonian~$\Hg_{0}$ on~$\Hh_0$ and 
an observable~$\Ag$ on~$\Hh$ satisfying Assumption~\ref{assu:NonDiag}.
The initial state~$\sigma_0$ lives in~$\Ss(\Hht)$ and 
we consider  the random sequences $(\sigma_k)_{k \in \NN}$ and $(\widetilde{\sigma}_k)_{k \in \NN}$ as in Section~\ref{sec:rep_measurement_model},
but where the unitary operator now alternates according to the parity of~$k$:
$$
\Ug_k :=
\begin{cases}
\Ug^+, & \text{if } k \text{ is even}, \\
\Ug^-, & \text{if } k \text{ is odd},
\end{cases}
$$
for two fixed unitaries $\Ug^-, \Ug^+$.
The following theorem is the main result of this alternating-operator mechanism and echoes
Theorem~\ref{thm:QuantumTraj_Disc}.

\begin{theorem}\label{thm:QuantumTraj_Modified}
For all $k \geq 0$, conditionally on $\{\widetilde{\sigma}_k = \theta_k\}$, 
the probability of observing the eigenvalue~$\lambda_j$ at step $k+1$ is
$$
\PP[\widetilde{\sigma}_{k+1} = \theta_j \mid \widetilde{\sigma}_k = \theta_k] = \Tr\left[(\Ig \otimes \Pg_j^{k+1}) \Ug_{k} (\theta_k \otimes \beta) \Ug_{k}^\dagger (\Ig \otimes \Pg_j^{k+1})\right],
$$
where the post-measurement state is given by
$$
\theta_j = \frac{\Pg_j^{k+1} \Ug_{k+1} (\theta_k \otimes \beta) \Ug_{k+1}^\dagger  \Pg_j^{k+1}}
{\Tr\left[\Pg_j^{k+1} \Ug_{k+1} (\theta_k \otimes \beta) \Ug_{k+1}^\dagger  \Pg_j^{k+1}\right]},
\qquad\text{for each }j=1,\ldots, \mf.
$$
\end{theorem}

Define now the sequence $(\rho_{k})_k:=(\EE_0[\sigma_k])_k$, similarly to~\eqref{eq:discrete_trajectory_single_U} as well as the states $\rho^{0},\rho^{1}$
as in~\eqref{eq:L0 and L1}, using~$\Ug^+$ and~$\Ug^-$ respectively as well as the probabilities~$p_{k+1}$, $q_{k+1}$ and the normalised random variable $X_{k+1}$ as in~\eqref{eq:def_X_k}, 
adapted to the present sequence~$(\Ug_k)$. 
It is straightforward to see that the discrete dynamics~\eqref{eq:discrete_evolution} still hold, namely
$$
\rhopm_{k+1} = 
\rhokz + \rhoko
 - \left(
\sqrt{\frac{q_{k+1}}{p_{k+1}}} \rhokz
- \sqrt{\frac{p_{k+1}}{q_{k+1}}} \rhoko
\right) X_{k+1},
\qquad\text{for all }k\geq 0.
$$

Given two operators $\Cg^+, \Cg^-$ on~$\Hh_0$,
define the sequence $(\Ug_k(n))_k$ by
$\Ug_{2k}(n)=\Ug^+(n)$ and $\Ug_{2k+1}(n)=\Ug^-(n)$ 
where $\Ug^{\pm}(n)=\sum_{i,j=0}^1\Ug_{ij}^{\pm}(n)\otimes e_i e_j^\dagger$ 
are unitary operators on~$\Hht$,
and assume as in~\eqref{eq:L00} that, as~$n$ tends to infinity,
\begin{equation}
\Ug_{00}^{\pm}(n) = \Ig - \frac{1}{n} \left(\I \Hg_{0} + \half \Cg^{\pm} (\Cg^{\pm})^\dagger \right) + \oon
\qquad\text{and}\qquad
\Ug_{10}^{\pm}(n) = \frac{1}{\sqrt{n}} \Cg^{\pm} + \oon.
\end{equation} 
We write $\Lg^{+}$ and $\Lg^{-}$ the two Lindblad generators linked to $\Cg^+$ and $\Cg^-$ and $(\rhopm_k(n))_k$ the random sequence of states on $\Hh_0$ constructed by repeated quantum interactions (with $\Ug^+(n)$ and $\Ug^-(n)$) and measurements. 
Under Assumption~\ref{assu:NonDiag}, 
Proposition~\ref{prop:increment_of_rho_k} is then updated as
\begin{subequations}
\begin{align}
\rhopm_{2k+1} - \rhopm_{2k}
& = \frac{\Theta^{-}(\rhopm_{2k})
+ o(1)}{\sqrt{n}} X_{2k+1} + \frac{1}{n} \Lg^{-}(\rhopm_{2k}) + \oon,\label{eq :alternating increment of rho_k 1}\\
\rhopm_{2k+2} - \rhopm_{2k+1}
& = \frac{\Theta^{+}(\rhopm_{2k+1})
+ o(1)}{\sqrt{n}} X_{2k+2} + \frac{1}{n} \Lg^{+}(\rhopm_{2k+1}) + \oon,\label{eq :alternating increment of rho_k 2}
\end{align}
\end{subequations}
for all $k\geq 0$, where $\Theta^{\pm}(\rho) := \Cg^\pm \rho + \rho (\Cg^\pm)^\dagger
- \Tr\left[ \rho(\Cg^\pm + (\Cg^\pm)^\dagger) \right] \rho$.

\begin{remark}
Similarly to Proposition~\ref{prop:increment_of_rho_k}, 
a parameter~$\gamma$
(depending on~$\Ag$ and  $\{e_{0},e_{1}\}$) has emerged  in~\ref{eq :alternating increment of rho_k 1} and~\ref{eq :alternating increment of rho_k 2}, 
and we normalise it  to~$1$ as in Remark~\ref{rem:gamma1}.
\end{remark}

\begin{theorem}\label{Convergence in E with two Cs}
Under Assumption~\ref{assu:NonDiag}, 
the sequence $\left(t \mapsto \EE[\rho_{\lfloor nt \rfloor}(n)]\right)_{n\geq 1}$ 
converges in $L^\infty([0, 1])$ as~$n$ tends to infinity to 
the unique solution of the ODE
$$
\frac{\D \phi_t}{\D t} = 
\frac{\Lg^{+}(\phi_t) + \Lg^{-}(\phi_t)}{2}, 
\quad\text{starting from }\phi_0 = \rho_0 \in \Ss(\Hh_0).
$$
\end{theorem}

\begin{proof}
Fix $k\geq 1$.
Taking expectation in~\eqref{eq :alternating increment of rho_k 1} and~\eqref{eq :alternating increment of rho_k 2},
writing $\rhob_k := \EE[\rhopm_{k}]$
and using the fact that the process~$(X_k)$ is centered conditionally to $\Ff_k$, we obtain
$$
\rhob_{2k+1} - \rhob_{2k}
= \frac{1}{n} \Lg^{-}(\rhob_{2k}) + \oon
\qquad\text{and}\qquad
\rhob_{2k+2} - \rhob_{2k+1}
= \frac{1}{n} \Lg^{+}(\rhob_{2k+1}) + \oon.
$$
Since $\Lg^{+}$ and $\Lg^{-}$ are linear and continuous, then
\begin{align*}
\rhob_{2k+2} & = \rhob_{2k} + \frac{1}{n}\Big(\Ug^{-}(\rhob_{2k}) + \Lg^{+}(\rhob_{2k}) + \Lg^{+}(\rhob_{2k+1}-\rhob_{2k})\Big) + \oon,\\
 & = \rhob_{2k} + \frac{1}{n}\Big(\Lg^{+}(\rhob_{2k}) + \Lg^{-}(\rhob_{2k})\Big) + \oon.
\end{align*}
We note $\phi$ the solution of the ODE with the initial condition $\phi_0=\rho_0$. Then
$$
\phi_{\frac{2k+2}{n}} = \phi_{\frac{2k}{n}}+\frac{1}{n}\Big(\Lg^{+}(\phi_{\frac{2k}{n}})+\Lg^{-}(\phi_{\frac{2k}{n}})\Big) + \oon,
$$
which leaves us with 
$$
\rhob_{2k+2} - \phi_{\frac{2k+2}{n}}
= \left(\rhob_{2k} - \phi_{\frac{2k}{n}}\right)
+ \frac{1}{n} \left(
\Lg^{+}(\rhob_{2k} - \phi_{\frac{2k}{n}}) +
\Lg^{-}(\rhob_{2k} - \phi_{\frac{2k}{n}})
\right)
+ o\left( \frac{1}{n} \right).
$$
Let $K \in \RR^+$ such that for all $\rho\in\Ss(\Hh_0)$, $\lVert \Lg^{+}(\rho) \rVert + \lVert \Lg^{-}(\rho) \rVert \leq K \lVert \rho \rVert$. 
Let $v_k:=\lVert \rhob_{2k}-\phi_{\frac{2k}{n}}\rVert$;
then $v_{k+1}\leq (1+\frac{K}{n})v_k+\frac{\alpha_n}{n}$ where $(\alpha_n)$ converges to zero.
Gronwall's lemma~\cite{clark1987gronwall} then implies that 
$v_k \leq \frac{\alpha_n}{K}(1+\frac{K}{n})^k$ for all~$k$.
Choosing $k=\lfloor \frac{nt}{2}\rfloor$, since
$(1+\frac{K}{n})^k$ converges to $\E^{\frac{Kt}{2}}$, 
the supremum $v_{\lfloor nt \rfloor}$ can be controlled and converges to zero, proving the theorem.
\end{proof}

Using~\eqref{eq :alternating increment of rho_k 1} and~\eqref{eq :alternating increment of rho_k 2} and denoting $\phi_k(n) := \rhopm_{2k}(n)$, we obtain
\begin{align*}
    \phi_{k+1}(n)
    &= \phi_k(n)+\frac{1}{n}\Big(\Lg^{+}\left(\phi_k(n)\right)
     + \Lg^{-}\left(\phi_k(n)\right)\Big) + \oon\\
    &+ \frac{1}{\sqrt{n}}\Big\{\Big(\Theta^{+}\left(\phi_k(n)\right)+o(1)\Big)X_{2k+1}
     + \Big(\Theta^{-}\left(\phi_k(n)\right)+o(1)\Big)X_{2k+2}\Big\}.
\end{align*}
In particular, setting $k=\lfloor \frac{nt}{2}\rfloor$, we have
\begin{align*}
\phi_{\lfloor \frac{nt}{2}\rfloor}(n)
 &= \rho_0 + \sum_{i=0}^{\lfloor \frac{nt}{2}\rfloor-1}\frac{1}{n}\Big(\Lg^{+}(\phi_i(n))+\Lg^{-}(\phi_i(n))\Big) +\oon\\
 &+ \sum_{i=0}^{\lfloor \frac{nt}{2}\rfloor-1}\left[
 \Big(\Theta^{+}(\phi_i(n)+o(1)\Big)\frac{X_{2k+1}}{\sqrt{n}}+\Big(\Theta^{-}(\phi_i(n))+o(1)\Big)\frac{X_{2k+2}}{\sqrt{n}}\right].
\end{align*}
Introduce the processes
\begin{equation*}
\begin{array}{rlr@{\;}ll}
W_n^{-}(t) & := \displaystyle \frac{1}{\sqrt{n}}\sum_{k=0}^{\lfloor \frac{nt}{2}\rfloor-1}X_{2k+2},
 & \qquad\displaystyle V_n(t) & := \displaystyle \frac{\lfloor \frac{nt}{2}\rfloor}{n}, \\
W_n^{+}(t) & := \displaystyle \frac{1}{\sqrt{n}}\sum_{k=0}^{\lfloor \frac{nt}{2}\rfloor-1}X_{2k+1},
& \qquad\phi_n(t) & := \displaystyle \phi_{2\lfloor \frac{nt}{2}\rfloor}(n),
\end{array}
\end{equation*}
so that 
\begin{align}
\phi_n(t)
 & =  \rho_0+\eps_n(t)+\int_0^t \Big(\Lg^{+}(\phi_n(s^-))+\Lg^{-}(\phi_n(s^-)\Big)\D V_n(s)\notag\\
 & \qquad\qquad\quad + \int_0^t\Theta^{+}(\phi_n(s^-))\D W_n^{+}(s)+\int_0^t\Theta^{-}(\phi_n(s^-))\D W_n^{-}(s),\notag \\
  & = \rho_0 + \eps_n(t)
 + \int_0^t \Big(\Lg^{+}(\phi_n(s^-))+\Lg^{-}(\phi_n(s^-)\Big)\D V_n(s)
 + \int_0^t\Theta(\phi_n(s^-))\cdot\D W_n(s)\label{eq : phi_n t U alter},
\end{align}
with~$\eps_n$ defined in an obvious way,
$\Theta(\rho):=\left ( \Theta^{+}(\rho), \Theta^{-}(\rho) \right )$ and $W_n(t):=\left ( W_n^{+}(t), W_n^{-}(t) \right )^{\top}$. 
The function~$\phi_n$ is c\`adl\`ag and  the integrals involving \( \phi_n(s^-) \) are well defined as Riemann–Stieltjes integrals.
The following are required to prove the convergence of $(W_n)$.

\begin{definition}
Let $B\in\Ss_d^+(\RR)$. 
We call~$M$ a $(0,B)$-Brownian motion if $M=\widetilde{B}W$ 
for some standard Brownian motion~$W$ in $\RR^d$ with $\widetilde{B}\in\Ss_m^+(\RR)$ such that $\widetilde{B}^2=B$.
\end{definition}

\begin{theorem}[Theorem~2.1 in~\cite{whitt2007martingale}]\label{th:whitt}
    For $n\in\NN^*$, let $M_n=(M_{n,1}, \dots,M_{n,d})^{\top}$ be a local martingale in $\DD([0,1],\RR^d)$ starting from~$0$ for each~$n$, and $B\in \Ss_d^+(\RR)$. 
    If
    $$
\lim_{n\uparrow\infty}\EE\left[\sup_{0\leq t\leq 1 } \lVert M_n(t)-M_n(t^-)\rVert_1\right]=0
$$
and 
$[M_{n,i},M_{n,j}]_t \Rightarrow tB_{ij}$ for all $i,j\in\{1,\dots,d\}$, $t\in [0,1]$,
then $(M_n) \Rightarrow M$ in $\DD([0,1],\RR^d)$.
\end{theorem}

We can now state the desired convergence, using the filtration $\Ff^n_t := \sigma(X_i|i\leq 2\lfloor\frac{nt}{2}\rfloor)$.

\begin{proposition}\label{cv of hat W_n}
The sequence $(W_n)_n$ converges in distribution to $\frac{1}{\sqrt{2}}W$ where $W$ is a two-dimensional Brownian motion.
\end{proposition}

The following simple result will be needed in the proof of the proposition:
\begin{lemma}\label{lemma:sup_EE_X_k^4_with_alter_U}
Under Assumption~\ref{assu:NonDiag}, 
$\sup_{k\in \NN^*}\EE[X_k^4]$ is finite.
\end{lemma}

\begin{proof}
Let $k\in\NN$ and $\Ff_k:=\sigma(X_l|l\leq k)$. 
With~\eqref{eq : pk+1 and qk+1 with U alter}, 
$$
\EE[X_{k+1}^4|\Ff_k] = \frac{p_{k+1}^2}{q_{k+1}}+\frac{q_{k+1}^2}{p_{k+1}}
 = \frac{\pf_{00}^2}{\qf_{00}}+\frac{\qf_{00}^2}{\pf_{00}}  + \Oo\left(\frac{1}{\sqrt{n}}\right),
$$
where the $\Oo\left(\frac{1}{\sqrt{n}}\right)$ term is uniform in~$k$.
\end{proof}

\begin{proof}[Proof of Proposition~\ref{cv of hat W_n}]
We first check that $(W_n^{+})$ and $(W_n^{-})$ are $(\Ff^n)$-martingales.
Recall that for all $k\in \NN$, with $\Ff_k=\sigma(X_i|i\leq k)$,
$\EE[X_{k+1}|\Ff_k]=0$
and $\EE[X_{k+1}^2|\Ff_k]=1$.
Therefore,
$$
\EE\left[W_n^{+}(t)|\Ff^n_s\right] = W_n^{+}(s)+\sum_{k=\lfloor \frac{ns}{2}\rfloor }^{\lfloor \frac{nt}{2}\rfloor-1}\EE\left[X_{2k+1}|\Ff^n_s\right] = W_n^{+}(s),
\qquad\text{for all }0\leq s\leq t,
$$
and $\EE\left[W_n^{+}(t)^2\right] = \frac{\lfloor\frac{nt}{2}\rfloor}{n}\leq \frac{t}{2}$.
Regarding the second condition of Theorem~\ref{th:whitt},
$$
\left[W_n^{+}, W_n^{+}\right]_t = \frac{1}{n}\sum_{k=0}^{\lfloor \frac{nt}{2}\rfloor-1} X_{2k+1}^2
\qquad\text{and}\qquad \left[W_n^{-}, W_n^{-}\right]_t = \frac{1}{n}\sum_{k=0}^{\lfloor \frac{nt}{2}\rfloor-1} X_{2k+2}^2,
$$
for any $t\in [0,1]$,
so that
\begin{align*}
& \EE\left[\left([W_n^{+}, W_n^{+}]_t-\frac{\lfloor \frac{nt}{2}\rfloor}{n}\right)^2\right]\\
& = \frac{1}{n^2}\left\{\sum_{k=0}^{\lfloor \frac{nt}{2}\rfloor-1}\EE\left[\left(X_{2i+1}^2-1\right)^2\right]
+ \sum_{i,j=0,i\neq j}^{\lfloor \frac{nt}{2}\rfloor-1}
\EE\left[\left(X_{2i+1}^2-1\right)\left(X_{2j+1}^2-1\right)\right]\right\} \leq C \frac{\lfloor \frac{nt}{2}\rfloor}{n^2},
\end{align*}
where $C:=\sup_{m\in \NN}\EE[(X_m^2-1)^2]$ is well defined by Lemma~\ref{lemma:sup_EE_X_k^4_with_alter_U}.
Thus
$([W_n^z, W_n^z]_t)_n$ converges in~$L^2$ (hence in distribution) to $\frac{t}{2}$ for $z\in\{-,+\}$.
Now, 
$[W_n^{+},W_n^{-}]_t=\frac{1}{n}\sum_{k=0}^{\lfloor \frac{nt}{2}\rfloor-1}X_{2k+1}X_{2k+2}$,
which implies that
\begin{align*}
\lim_{n\uparrow \infty}
\EE\left[[W_n^{+},W_n^{-}]_t^2\right]
 & = 
\lim_{n\uparrow \infty}
\frac{1}{n^2}\left\{\sum_{k=0}^{\lfloor \frac{nt}{2}\rfloor-1} \EE\left[X_{2k+1}^2 X_{2k+2}^2\right]
  + \sum_{i,j=0, i\ne j}^{\lfloor \frac{nt}{2}\rfloor-1}\EE[X_{2i+1}X_{2i+2}X_{2j+1}X_{2j+2}]\right\}\\
   & = \lim_{n\uparrow \infty}\frac{\lfloor \frac{nt}{2}\rfloor}{n^2} = 0,
\end{align*}
since $\EE[X_{2j+2}|\Ff_{2j+1}]=0$ for all $j$.
Regarding the first condition in Theorem~\ref{th:whitt}, 
the only possible jumps in~$W_n$ are the terms $\delta W_{n,k} := \frac{1}{\sqrt{n}}(X_{2k+1}, X_{2k+2})^{\top}$, so that
$$
\EE\left[\sup_{0\leq t\leq 1} \left\lVert W_n(t)-W_n(t^-)\right\rVert_1\right]
 = \EE\left[\sup_{0\leq k \leq \lfloor\frac{n}{2}\rfloor-1} \left\lVert \delta W_{n,k}\right\rVert_1\right].
 $$
Now, by convexity and with Lemma~\ref{lemma:sup_EE_X_k^4_with_alter_U},
\begin{align*}
\EE\left[\left(\sup_{0\leq k \leq \lfloor\frac{n}{2}\rfloor-1} \left\lVert \delta W_{n,k}\right\rVert_1\right)^4\right]
 & = \frac{1}{n^2}\EE\left[\sup_{0\leq k \leq \lfloor\frac{n}{2}\rfloor-1} \Big(|X_{2k+1}|+|X_{2k+2}|\Big)^4\right] \\
 &\leq \frac{8}{n^2}\EE\left[\sup_{0\leq k \leq \lfloor\frac{n}{2}\rfloor-1} \Big(|X_{2k+1}|^4+|X_{2k+2}|^4\Big)\right]\\
&\leq \frac{16\lfloor\frac{n}{2}\rfloor}{n^2}\sup_{k\in\NN^*}\EE[X_k^4].
\end{align*}
Therefore $\lim\limits_{n\uparrow\infty}\EE[\sup_{0\leq t\leq 1} \lVert W_n(t)-W_n(t^-)\rVert_1]=0$
and the proposition follows.
\end{proof}

Before stating the final theorem of this section, we show that the expected limiting process of our sequence of states is well defined.
\begin{lemma}\label{lemma : uniqueness of solution with alter U}
Let $W$ a two-dimensional Brownian motion.
The SDE
\begin{equation}\label{eq:SDE_alter_U}
\rhopm_t = \rho_0+\half \int_0^t 
\Big(\Lg^{+}(\rhopm_s)+\Lg^{-}(\rhopm_s)\Big)\D s
+ \frac{1}{\sqrt{2}}\int_0^t \Theta(\rhopm_s)\cdot\D W(s),
\qquad\text{for }t \in [0,1],
\end{equation}
with $\rho_0 \in \Ss(\Hh_0)$,
admits a unique strong solution in $\Ss(\Hh_0)$.
\end{lemma}

\begin{proof}
$\Lg^{+}$ and $\Lg^{-}$ are linear continuous maps on $\Ss(\Hh_{0})$, $\Theta^{+}$ and $\Theta^{-}$ are quadratic continuous. But for all $\rho\in \Ss(\Hh_{0})$, using that $\lVert \rho \rVert \leq 1$, we obtain
$$
\lVert \Theta^{+}(\rho)\rVert \leq \lVert \rho \rVert \left(\lVert \Cg^+ \rVert+ \lVert (\Cg^+)^\dagger \rVert\right) + \left|\Tr\left[\rho(\Cg^++(\Cg^+)^\dagger)\right]\right| \leq K(1+\lVert \rho \rVert),
$$
where $K$ is a positive constant that depends of~$\Cg^+$. 
Hence, all four functions satisfy at most linear growth on~$\Ss(\Hh_{0})$. 
We already know~$\Lg^{+}$ and~$\Lg^{-}$ are Lipschitz, and~$\Theta^{+}$ and~$\Theta^{-}$ are locally Lipschitz (since they are in $\Cc^\infty$). 
By compactness of $\Ss(\Hh_0)$, the four functions are Lipschitz 
and the lemma follows from~\cite[Theorem~5.2.1]{oksendal2003sde}.
\end{proof}

\begin{theorem}\label{th: cv SDE with alternating U}
Under Assumption~\ref{assu:NonDiag}, 
$(\rhopm_{2\lfloor\frac{n\cdot}{2}\rfloor}(n))_n$ converges in distribution to~\eqref{eq:SDE_alter_U}.
\end{theorem}
The proof relies on the following, which ensures that convergence of the integrands leads to convergence of the solutions of the corresponding SDEs.
\begin{proposition}[Proposition~1 in~\cite{kurtz1991wong}]\label{prop:KurtzWong}
    Let $f : \RR^{d} \xrightarrow{}\Mm_{d,d'}(\RR)$ be bounded and continuous,  $(U_n, Y_n)$ be an $(\Ff^n)$-adapted process in $\DD([0,1],\RR^{d}\times\RR^{d'})$ and $(Y_n)$ a \textit{good} sequence of semimartingales with $(U_n, Y_n)\Rightarrow (U,Y)$.
    If for each $n$, $X_n$ is the solution to
    $X_n(t) = U_n(t)+\int_0^tf(X_n(s^-))\D Y_n(s)$,
    then the sequence ${(X_n,Y_n)}$ is tight and any limit point X satisfies $X(t)=U(t)+\int_0^tf(X(s^-))\D Y(s)$.
\end{proposition}
This proposition requires the notion of a \textit{good} sequence~\cite[Proposition~1.5]{memin_slominski_stability}:
Let $(Y_n)_n$ be a sequence of semimartingales with values in $\DD([0,1],\RR^d)$ with decomposition $Y_n = M_n + A_n$ with~$M_n$ a local martingale 
and~$A_n$ a finite-variation process. 
The sequence~$(Y_n)_n$ is called \textit{good} if for each $\alpha>1$, there exist stopping times $\{\tau^\alpha_n\}_n$ with $\PP(\tau^\alpha_n\geq \alpha)\leq \frac{1}{\alpha}$ such that
\begin{equation}\label{def of goodness}
\sup_{n\in\NN}\EE\Big[[M_{n,i},M_{n,i}]_{t\wedge\tau^\alpha_n}+T_t(A_{n,i})\Big]\text{ is finite for all }i\in\{1,\dots,d\},
\quad\text{for all }t>0,
\end{equation}
where $T_t(A_{n,i})$ is the total variation of~$A_{n,i}$ up to time~$t$.
We can now prove Theorem~\ref{th: cv SDE with alternating U}.

\begin{proof}[Proof of Theorem~\ref{th: cv SDE with alternating U}]
The process $Y_n:=(V_n,W_n^{+}, W_n^{-})^{\top}$ is a $(\Ff^n)$-semimartingale 
since~$V_n$ is of finite variation and~$W_n^{+}, W_n^{-}$ are two $(\Ff^n)$-martingales.
With Proposition~\ref{cv of hat W_n}, $(\eps_n+\rho_0, Y_n) \Rightarrow (\rho_0,Y)$ where $Y(t)=(\frac{t}{2}, \frac{1}{\sqrt{2}}B^o(t), \frac{1}{\sqrt{2}}B^e(t))^{\top}$ and $(B^o, B^e)^{\top}$ is a two-dimensional Brownian motion.
The map $f(\rho):=(\Lg^{+}(\rho)+\Lg^{-}(\rho), \Theta^{+}(\rho), \Theta^{-}(\rho))$ 
is continuous and bounded on~$\Ss(\Hh_{0})$. 
The proof thus follows from Proposition~\ref{prop:KurtzWong} by checking the \textit{goodness} condition~\eqref{def of goodness}.
For all $\alpha>1$, let $(\tau^\alpha_n)_n=(\frac{1+\alpha}{2})_n$, 
so that $t\wedge \tau^\alpha_n=t$. 
With Proposition~\ref{cv of hat W_n}, a decomposition of the semimartingale~$Y_n$ is 
$Y_n = M_n + A_n$ where $M_n=(0, W_n^{+}, W_n^{-})^{\top}$ 
and $A_n=(V_n, 0, 0)^{\top}$. We also proved with Proposition~\ref{cv of hat W_n} that $([W_n^i, W_n^i]_t)$ converges in $L^1(\RR)$ for all $i\in\{+,-\}$. 
Hence 
$\sup_{n\in\NN}\EE[[M_{n,i},M_{n,i}]_{t\wedge \tau^\alpha_n}]$ is finite for all $i\in\{1,2,3\}$. 
Also, the only jumps in $A_n$ are the terms $\frac{1}{n}(1, 0, 0)$ and occur $\lfloor\frac{nt}{2}\rfloor$ times up to time~$t$. 
Then 
$T_t(A_n)=(T_t(A_{n,i}))_{i\in\{1,2,3\}}=
\begin{pmatrix}
\frac{\lfloor\frac{nt}{2}\rfloor}{n}& 0 & 0
\end{pmatrix}$,
and therefore
$\sup_{n\in\NN}\EE[T_t(A_{n,i})]\leq \frac{t}{2}$
for all $i\in\{1,2,3\}$.
Then, with Proposition~\ref{prop:KurtzWong}, $(\rho_{2\lfloor\frac{nt}{2}\rfloor}(n))_n$ converges in distribution to the unique strong solution~$\rho$ to~\eqref{eq:SDE_alter_U}. 
Since~$\Ss(\Hh_{0})$ is closed, so it 
$\DD([0,1],\Ss(\Hh_{0}))$ and clearly $(\rho_{2\lfloor\frac{nt}{2}\rfloor}(n))_{t\geq0} \in \DD([0,1],\Ss(\Hh_{0}))$
for all $n\in \NN$, so that $(\rho_t)_{t\geq0} \in \DD([0,1],\Ss(\Hh_{0}))$.
\end{proof}

\subsection{Robustness under unitary perturbation}\label{section 2b}
We now investigate how the discrepancy between two unitary operators—one exact, the other perturbed—impacts the evolution of quantum trajectories. 
Specifically, we analyse the difference between the solutions of the stochastic differential equations arising from two repeated interaction and measurement models, each associated with one of the respective unitaries.
From a practical standpoint, this approach is well motivated: in quantum information processing, the application of an ideal unitary~$\Ug$ on a system is affected (especially in current NISQ hardware technology) by imperfections, so that the implemented operation is in fact a noisy channel close--but not equal--to the noiseless unitary channel $\rho \mapsto \Ug\rho \Ug^\dagger$.
We assume that $\Hh_0=\CC^2$ and $\Hh=\CC^8$, 
take $e_{0}\in \CC^4$ be a unitary vector and set $\beta=e_{0}e_{0}^\dagger$. 
We complete~$e_{0}$ into an orthonormal basis $\{e_0,\ldots,e_3\}$
and let~$\Ag$ be an observable on~$\CC^4$. 

In quantum computing, the notion of \emph{noise} is usually 
introduced via quantum channels, which are convex, linear, completely positive and trace preserving operators on the space~$\Ss(\Hh_{0})$ of density operators on~$\Hh_{0}$.
We refer the reader to~\cite[Section~8.2]{nielsen_chuang_2000} for full details,
but recall in particular the following result by Kraus~\cite{kraus1983states},
based on the Choi decomposition~\cite{choi1975completely}:
\begin{theorem}
Any quantum operation~$\Ee$ on a system of Hilbert space~$\Hh$ of dimension~$d$ has the form
$\Ee(\rho) = \sum_{k=0}^{r-1}\Kg_i \rho \Kg_i^\dagger$,
where $(\Kg_i)_{i=0,\ldots,r-1}$ is a collection of (Kraus) operators on~$\Hh$ such that $\sum_{i=0}^{r-1}\Kg_i^\dagger \Kg_i = \Ig$ and $1\leq r\leq d^2$ is called the Kraus rank.
\end{theorem}

\begin{example}[Noise models]
The following are standard noise models~\cite[Section~8.3]{nielsen_chuang_2000} 
(where $\Xg, \Yg,\Zg$ are the usual Pauli gates) with $p\in [0,1]$:
\begin{enumerate}[(a)]
\item \textit{Bit flip channel:} 
$r=1$, $\Kg_{0} = \sqrt{p}\,\Ig$, $\Kg_{1} = \sqrt{1-p}\,\Xg$, 
so that $\Ee(\rho) = p \rho + (1-p)\Ig$;
\item \textit{Phase flip channel:} 
$r=1$, $\Kg_{0} = \sqrt{p}\,\Ig$, $\Kg_{1} = \sqrt{1-p}\,\Zg$, 
so that $\Ee(\rho) = p \rho + (1-p)\Zg\rho\Zg^\dagger$;
\item \textit{Bit-phase flip channel:} 
$r=1$, $\Kg_{0} = \sqrt{p}\,\Ig$, $\Kg_{1} = \sqrt{1-p}\,\Yg$, 
so that $\Ee(\rho) = p \rho + (1-p)\Yg\rho\Yg^\dagger$;
\item \textit{Completely depolarising channel:} 
$r=4$, $\Kg_{0} = \sqrt{1-\frac{3p}{4}}\,\Ig$, $\Kg_{1} = \frac{\sqrt{p}}{2}\Xg$, 
$\Kg_{2} = \frac{\sqrt{p}}{2}\Yg$, $\Kg_{3} = \frac{\sqrt{p}}{2}\Zg$, 
so that $\Ee(\rho) = (1-p)\rho+\frac{p}{2}\Ig$.
\end{enumerate}
\end{example}

We therefore consider from now on a noise model of the form 
\begin{equation}\label{eq:NoiseModelChannel}
\Ee(\rho) = \Omega^\eps(\rho)
:= (1-\eps) \rho +  \eps \Omega(\rho),
\qquad\text{for }\rho \in \Ss(\Hh_0),
\end{equation}
for $\eps\in [0,1]$, where $\Omega$ is a given quantum channel,
and hence so is~$\Omega^\eps$. 
We wish to find a unitary operator~$\Ug^{\eps}(n)$
on $\Hh_0\otimes\Hh$ such that the sequence $(\rho_{k}^{\eps}(n))$ (resulting from the repeated interactions and  measurements with $\Ug^{\eps}(n)$ and~$\Ag$) verifies 
$$
\lim_{n\uparrow\infty}
\sup_{0\leq t \leq 1} \left\|\EE\left[\rho^\eps_{\lfloor nt \rfloor}(n)\right] -\phi_t^\eps\right\|=0,
\qquad\text{with } \phi_t^\eps=\rho_0+\int_0^t \Omega^\eps(\phi_s)\D s, \text{ for }t \in [0,1].
$$

\begin{lemma}\label{lemma : U eps as wanted}
Let $\displaystyle \Ug^{\eps}(n)=\sum_{i,j=0}^{3}\Ug_{ij}\otimes  e_i e_j^\dagger$ be such that
\begin{equation}\label{Lij blocks with U noise}        
\Ug_{00} = \left(1+\frac{1-\eps}{2n}\right)\Kg_{0}
\qquad\text{and}\qquad
\Ug_{i0} = 
\sqrt{\frac{\eps}{n}} \Kg_i, \text{ for }i\in\{1,2,3\},
\end{equation}
for some operators $(\Kg_i)_{i=0,\ldots,3}$ with $\Kg_0=\Ig$
and the others~$\Ug_{ij}$ of any form as long as $\Ug^\eps(n)$ is unitary.
Then, for any $\rho\in\Ss(\Hh_0)$,
$$
\EE_0\left[\Ug^{\eps}(n)(\rho\otimes\beta)\Ug^{\eps}(n)^\dagger\right]
= \rho + \frac
{1}{n}\Omega^\eps(\rho)+\oon,
\quad \text{as~$n$ tends to infinity},
$$
where~$\Omega^{\eps}$ has the form~\eqref{eq:NoiseModelChannel}
with 
$\Omega(\rho) = \sum_{i=1}^{3} \Kg_i \rho \Kg_i^\dagger$.
\end{lemma}
\begin{remark}
We have not assumed that $(\Kg_i)_{i=0,\ldots,3}$ are Kraus operators, 
though of course the Kraus decomposition clearly offers such a choice.
\end{remark}

\begin{proof}
For any $\rho\in\Ss(\Hh_0)$, $\Ug^{\eps}(n)(\rho\otimes \beta)\Ug^{\eps}(n)^\dagger=\sum_{i,j=0}^{3}\Ug_{i0}\rho \Ug_{j0}^\dagger\otimes
e_i e_j^\dagger$
and therefore 
$$
\EE_0[\Ug^{\eps}(n)(\rho\otimes \beta)\Ug^{\eps}(n)^\dagger]
= \sum_{i=0}^{3} \Ug_{i0}\rho \Ug_{i0}^\dagger
= \rho+\frac{1}{n}\left[(1-\eps)\rho + \eps\sum_{i=1}^{3} \Kg_i \rho \Kg_i^\dagger\right]  + \oon.
$$
\end{proof}

For $i\in\{0,\dots, 3\}$ and $\rho\in\Ss(\Hh_{0})$, define
$q_{k+1}^{(i)} := \Tr\left[\rhoieps_{k}(n)\right]$
and 
$p_{k+1}^{(i)} := 1 - q_{k+1}^{(i)}$,
where
$$
\rhoieps(n)
 := \EE_0\Big[(\Ig\otimes \Pg_i)\Ug^{\eps}(n)
\left(\rho\otimes\beta\right)
 \Ug^{\eps}(n)^\dagger(\Ig\otimes \Pg_i)\Big],
$$
as well as the random variable
$\nu_{k+1}^i(\omb)=\ind_{\{\omega_{k+1}=i\}}$ for any $\omb\in\{0,\dots,3\}^{\NN^*}$. Conditionally on~$\rho_{k}^{\eps}(n)$, we have 
$$
\nu_{k+1}^i =
\begin{cases}
0, & \text{with probability } p_{k+1}^{(i)}, \\
1, & \text{with probability } q_{k+1}^{(i)},
\end{cases}
$$
so that
$$
\rho_{k+1}^{\eps}(n)
= \left(1-\sum_{i=1}^{3}\nu_{k+1}^i\right)\frac{\rhokzeps(n)}{q_{k+1}^{(0)}}
 + \sum_{i=1}^{3}\nu_{k+1}^{i} \frac{\rhoieps_{k}(n)}{q_{k+1}^{(i)}}.
$$
Remark~\ref{rem:Proba01} applies again.
With $X_{k+1}^{(i)}$ the normalised version of $\nu_{k+1}^i$, 
we obtain:
\begin{proposition}\label{prop : inc rho eps k+1 U noise op2}
For any state $\rho_0\in\Ss(\Hh_{0})$ and any observable~$\Ag$ on $\CC^4$, 
\begin{equation}\label{discrete trajectory with U noise op2}
\rho_{k+1}^{\eps}(n)=\sum_{i=0}^{3}\rhoieps_{k}(n) + \sum_{i=1}^{3}
\left(\sqrt{\frac{p_{k+1}^{(i)}}{q_{k+1}^{(i)}}}\rhoieps_{k}(n) - \frac{\sqrt{p_{k+1}^{(i)}q_{k+1}^{(i)}}}{q_{k+1}^{(0)}}\rhokzeps(n)\right) X_{k+1}^{(i)}.
\end{equation}
\end{proposition}
We write the projectors $\Pg_i = (\pf_{kl}^{(i)})$. 
Using Lemma~\ref{lemma : U eps as wanted}, the increment of the sequence $(\rho_{k}^{\eps}(n))_k$ is rewritten as follows, under the assumption (analogous to Assumption~\ref{assu:NonDiag})
\begin{equation}\label{assu:r00}
\pf_{00}^{(i)}
\left(1-\pf_{00}^{(i)}\right)
\neq 0,
\qquad\text{for all }
i\in\{0,\dots,3\}.
\end{equation}

\begin{proposition}\label{increment of rho_eps with U nosie op2}
Assuming~\eqref{assu:r00}, then
\begin{equation}
\rho_{k+1}^{\eps}(n)
= \rho_{k}^{\eps}(n)+\frac{1}{n}\Omega^\eps(\rho_{k}^{\eps}(n)) + \frac{\sqrt{\eps}}{\sqrt{n}}\sum_{i=1}^{3}
\Big(\Xi_i(\rho_{k}^{\eps}(n)) + o(1)\Big) X_{k+1}^{(i)} + \oon,
\end{equation}
where $\Xi_i=\sqrt{\pf_{00}^{(i)}(1-\pf_{00}^{(i)})}\sum_{a=1}^{3}\Theta_{\Kg_a^{(i)}}$ with $\Kg_a^{(i)}=\gamma_a^{(i)}\Kg_a$ and $\gamma_a^{(i)}\in\CC$ (depending on~$i$, on~$\Ag$ and on the basis $\{e_j\}$) for all $a\in\{1,\dots,3\}$ and $i\in\{1,\dots,3\}$.
\end{proposition}

\begin{proof}
Write $\rho_{k}^{\eps}$ instead of $\rho_{k}^{\eps}(n)$.
For $i\in\{0,\dots,3\}$, we have
\begin{align}
\rhoieps_{k} & = \sum_{a,b=1}^{3}\pf_{ba}^{(i)}\Ug_{a0}\rho_{k}^{\eps} \Ug_{b0}^\dagger \notag\\
&=\pf_{00}^{(i)}\left(1+\frac{1-\eps}{n}\right)
\rho_{k}^{\eps}+\frac{\sqrt{\eps}}{\sqrt{n}}\sum_{a=1}^{3}
\left(\pf_{0a}^{(i)}\Kg_a \rho_{k}^{\eps}+\pf_{a0}^{(i)}\rho_{k}^{\eps}\Kg_a^\dagger\right)+\frac{\eps}{n}\sum_{a,b=1}^{3}\pf_{ba}^{(i)}\Kg_a\rho_{k}^{\eps}\Kg_b^\dagger +\oon \notag\\
&=\pf_{00}^{(i)}\rho_{k}^{\eps}+\frac{1}{\sqrt{n}}\Gamma_i
(\rho_{k}^{\eps})
 + \frac{1}{n}\left(\eps\sum_{a,b=1}^{3}\pf_{ba}^{(i)}\Kg_a\rho_{k}^{\eps}\Kg_b^\dagger +(1-\eps)
 \rho_{k}^{\eps}\right)
 +\oon \label{Ugb_i(rho_epsilon)_approximation_with_U_noise}
\end{align}
where $\Gamma_i(\rho)=\sqrt{\eps}\sum_{a=1}^{3}\left(\pf_{0a}^{(i)}\Kg_a \rho_{k}^{\eps}+\pf_{a0}^{(i)}\rho_{k}^{\eps}\Kg_a^\dagger\right)$.
Hence
$$
\sum_{i=0}^{3}\rhoieps_{k}
= \rho_{k}^{\eps}+\frac{1}{n}\left(\eps\sum_{a=1}^{3}\Kg_a\rho_{k}^{\eps} \Kg_a^\dagger
+ (1-\eps)\rho_{k}^{\eps}\right)
 + \oon
  = \rho_{k}^{\eps}+\frac{1}{n}\Omega^\eps(\rho_{k}^{\eps})+\oon.
$$

Denote $q^{(i)}:=\Tr[\Gamma_i(\rho_{k}^{\eps})]$. 
Then
\begin{equation}\label{approx of qk+1 and pk+1 with U noise}
q_{k+1}^{(i)} := \pf_{00}^{(i)}+\frac{q^{(i)}}{\sqrt{n}}+\oonsq
\qquad\text{and}\qquad
p_{k+1}^{(i)} := 1-\pf_{00}^{(i)}-\frac{q^{(i)}}{\sqrt{n}}.
\end{equation}
Let us focus on the terms $E_i:=\sqrt{\frac{p_{k+1}^{(i)}}{q_{k+1}^{(i)}}}\rhoieps_{k} - \frac{\sqrt{p_{k+1}^{(i)}q_{k+1}^{(i)}}}{q_{k+1}^{(0)}}\rhokzeps$ for $i\in\{1,\dots,3\}$. 
Denoting $\alpha^{(i)}:=\sqrt{\frac{1-\pf_{00}^{(i)}}{\pf_{00}^{(i)}}}$,
we can compute
\begin{align*}
\sqrt{\frac{p_{k+1}^{(i)}}{q_{k+1}^{(i)}}}
 & = \alpha^{(i)}\left(1+\frac{q^{(i)}}{2\pf_{00}^{(i)}(1-\pf_{00}^{(i)})\sqrt{n}}+\oonsq\right),\\
\frac{\sqrt{p_{k+1}^{(i)}q_{k+1}^{(i)}}}{q_{k+1}^{(0)}}
 & = \frac{\sqrt{\pf_{00}^{(i)}(1-\pf_{00}^{(i)})}}{\pf_{00}^{(0)}}
 \left[1+\frac{1}{\sqrt{n}}\left(\frac{q^{(i)}(1-\pf_{00}^{(i)})}{2\pf_{00}^{(i)}(1-\pf_{00}^{(i)})}-\frac{q^{(0)}}{\pf_{00}^{(0)}}\right)
 + \oonsq\right],
\end{align*}
so that
\begin{align*}
    E_i&=\alpha^{(i)}\left[1+\frac{q^{(i)}}{2\pf_{00}^{(i)}(1-\pf_{00}^{(i)})\sqrt{n}}+\oonsq\right]\left\{\pf_{00}^{(i)}\rho_{k}^{\eps}+\frac{1}{\sqrt{n}}\Gamma_i(\rho_{k}^{\eps})+\oonsq\right\}\\
    &-\frac{\sqrt{\pf_{00}^{(i)}(1-\pf_{00}^{(i)})}}{\pf_{00}^{(0)}}
    \left[1+\frac{1}{\sqrt{n}}\left(\frac{q^{(i)}(1-\pf_{00}^{(i)})}{2\pf_{00}^{(i)}(1-\pf_{00}^{(i)})}-\frac{q^{(0)}}{\pf_{00}^{(0)}}\right) + \oonsq\right]\left\{\pf_{00}^{(0)}\rho_{k}^{\eps}+\frac{1}{\sqrt{n}}\Gamma_0(\rho_{k}^{\eps})+\oonsq\right\}\\
    & = \frac{\beta_i}{\sqrt{n}}\left[\left(\frac{q^{(0)}}{\pf_{00}^{(0)}}-\frac{q^{(i)}}{\pf_{00}^{(i)}}\right)\rho_{k}^{\eps}+\frac{1}{\pf_{00}^{(i)}}\Gamma_i(\rho_{k}^{\eps})-\frac{1}{\pf_{00}^{(0)}}\Gamma_0(\rho_{k}^{\eps})+o(1)\right],
\end{align*}
where $\beta_i:=\sqrt{\pf_{00}^{(i)}(1-\pf_{00}^{(i)})}$.
For $a\in\{1,\dots,3\},i\in\{1,\dots,3\}$, we define $\gamma_a^{(i)}:=\frac{\pf_{0a}^{(i)}}{\pf_{00}^{(i)}}-\frac{\pf_{0a}^{(0)}}{\pf_{00}^{(0)}}$ and $\Kg_a^{(i)}:=\gamma_a^{(i)}\Kg_a$, then the last expression finally reads
$$
E_i=\frac{\sqrt{\eps}}{\sqrt{n}}\beta_i\left[\sum_{a=1}^{3}\left(\Kg_a^{(i)}\rho_{k}^{\eps}+\rho_{k}^{\eps} (\Kg_a^{(i)})^\dagger - 
\Tr\left[\rho_{k}^{\eps}\left(\Kg_a^{(i)}+(\Kg_a^{(i)})^\dagger\right)\right]\rho_{k}^{\eps}\right)+o(1)\right].
$$
\end{proof}

\begin{remark}
Since $\lVert \rho_{k}^{\eps}\rVert \leq 1$ for all $k\in \NN$, the term $\oonsq$ in~\eqref{Ugb_i(rho_epsilon)_approximation_with_U_noise} is bounded independently of $k$ and $\omb\in\Omega^{\NN^*}$. This independence follows also in the approximations~\eqref{approx of qk+1 and pk+1 with U noise}.
\end{remark}

\begin{lemma}\label{lem:def_bij_with_U_noise_op2}
Assuming~\eqref{assu:r00}, then, for all $i, j\in\{1,\dots,3\}$,
$k\in \NN$,
$$
K_i:=\sup_{k\in\NN} \EE\left[\left(X_{k+1}^{(i)}\right)^4\right]<\infty
\qquad\text{and}\qquad
\sqrt{\frac{q_{k+1}^{(i)}q_{k+1}^{(j)}}{p_{k+1}^{(i)}p_{k+1}^{(j)}}}=b_{ij}
\left[1 + \Oo\left(\frac{1}{\sqrt{n}}\right)\right],
$$
where $b_{ij}:=\sqrt{\frac{\pf_{00}^{(i)}\pf_{00}^{(j)}}{(1-\pf_{00}^{(i)})(1-\pf_{00}^{(j)})}}$ and $\Oo\left(\frac{1}{\sqrt{n}}\right)$ is bounded independently of $k$ and $\omb\in\Omega^{\NN^*}$.
\end{lemma}

\begin{proof}
Let $i\in\{1,\dots,3\}$, $k\in\NN$ and $\Ff_k=\sigma(X_l^{(j)}|j\in\{1,\dots,3\},l\leq k)$.
With~\eqref{approx of qk+1 and pk+1 with U noise} we have
$$
\EE\left[\left(X_{k+1}^{(i)}\right)^4|\Ff_k\right]=\frac{(p_{k+1}^{(i)})^2}{q_{k+1}^{(i)}}+\frac{(q_{k+1}^{(i)})^2}{p_{k+1}^{(i)}}=\frac{(1-\pf_{00}^{(i)})^2}{\pf_{00}^{(i)}}+\frac{(\pf_{00}^{(i)})^2}{1-\pf_{00}^{(i)}}
+ \Oo\left(\frac{1}{\sqrt{n}}\right),
$$
where the $\Oo\left(\frac{1}{\sqrt{n}}\right)$ term is bounded independently of~$k$.
Again, by using~\eqref{approx of qk+1 and pk+1 with U noise}
$$
\sqrt{\frac{q_{k+1}^{(i)}q_{k+1}^{(j)}}{p_{k+1}^{(i)}p_{k+1}^{(j)}}}
 = b_{ij}\left\{1+\frac{1}{\sqrt{n}}\left(\frac{q^{(i)}}{\pf_{00}^{(i)}(1-\pf_{00}^{(i)})}+\frac{q^{(j)}}{\pf_{00}^{(j)}(1-\pf_{00}^{(j)})}+\oonsq\right)
\right\}.
$$
Note that $q^{(i)}$ is bounded above independently of $k$ and $\omb\in\Omega^{\NN^*}$ (because $\lVert \rho_{k}^{\eps}(\omb)\rVert \leq 1$).
\end{proof}

We define the processes 
$\displaystyle V_n(t) := \frac{\lfloor nt\rfloor}{n}$
and $\displaystyle W_n^{i}(t) := \frac{1}{ \sqrt{n}}\sum_{k=0}^{\lfloor nt \rfloor-1}X_{k+1}^{(i)}$
as well as 
$$
\eps_n(t) := 
\rho_n^\eps(t) - \rho_0 - \int^t_0\Omega^\eps(\rho_n^\eps(s^-))\D V_n(s) - \sqrt{\eps}\int_0^t \Xi(\rho_n^\eps(s^-))\D W_n(s),
$$
with $\rho_n^\eps(t) := \rho^\eps_{\lfloor nt \rfloor}(n)$ 
and $\Xi(\rho) := (\Xi_i(\rho))_{i\in\{1,\dots,3\}}$.
Consider the filtration $(\Ff^n_t) = \sigma\left(X_k^{(i)}:i\in\{1,\dots,3\}, k\leq \lfloor nt \rfloor\right)$.

\begin{proposition}\label{cv of W_n with U noise}
For each $n$, $W_n$ is a $(\Ff^n)$-martingale
and $(W_n)_{n}$ converges in a distribution to a $(0,B)$-martingale 
where, with $b_{ij}$ defined in Lemma~\ref{lem:def_bij_with_U_noise_op2} and 
\begin{equation}\label{def of B}
B_{ij}=\begin{cases}
1, & \text{if } i=j,\\
-b_{ij},& \text{if } i\neq j.
\end{cases}
\end{equation}
\end{proposition}

\begin{proof}
That $(W_n)$ is a $(\Ff^n)$-martingale holds since, 
for all $k\in \NN$ and $i\in\{1,\dots,3\}$,
\begin{equation}\label{X_k centered with U noise}
\EE\left[X_{k+1}^{(i)}|\Ff_k\right]=0
\qquad\text{and}\qquad
\EE\left[\left(X_{k+1}^{(i)}\right)^2|\Ff_k\right]=1,
\end{equation}
where $\Ff_k=\sigma(X_l^{(i)}|j\in\{1,\dots,3\},l\leq k)$.
The proof is similar to that of Proposition~\ref{cv of hat W_n}.
We wish to apply Theorem~\ref{th:whitt} to $(W_n)$.
The only jumps possible in $W_n$ are the terms $\delta W_{n,k}=\frac{1}{\sqrt{n}}(X_{k+1}^{(i)})_{i\in\{1,\dots,3\}}$. By convexity and using Lemma~\ref{lem:def_bij_with_U_noise_op2},
$$
\EE\left[\left(\sup_{0\leq k \leq n-1} \lVert \delta W_{n,k}\rVert_1\right)^4\right]
=\frac{1}{n^2}\EE\left[\sup_{0\leq k \leq n-1} \left(\sum_{i=1}^{3}\left|X_{k+1}^{(i)}\right|\right)^4\right]
\leq \frac{3^3}{n}\sum_{i=1}^{3}K_i.
$$
so that the first condition of Theorem~\ref{th:whitt} holds.
Let now  $i,j\in\{1,\ldots,3\}$, $t\in[0,1]$; then
$[W_n^i, W_n^j]_t=\frac{1}{n}\sum_{k=0}^{\lfloor nt \rfloor-1}X_{k+1}^{(i)}X_{k+1}^{(j)}$.
If $i=j$, then
\begin{align*}
& \EE\left[\left([W_n^i, W_n^i]_t-\frac{\lfloor nt \rfloor}{n}\right)^2\right]\\
&=\frac{1}{n^2}\sum_{k=0}^{\lfloor nt \rfloor-1} \EE\left[\left((X_{k+1}^{(i)})^2-1\right)^2\right] + \frac{2}{n^2}\sum_{0\leq l<k\leq \lfloor nt \rfloor-1}
\EE\left[\left((X_{k+1}^{(i)})^2-1\right)\left((X_{l+1}^{(i)})^2-1\right)\right]\\
&=\frac{1}{n^2}\sum_{k=0}^{\lfloor nt \rfloor-1} \EE\left[\left((X_{k+1}^{(i)})^2-1\right)^2\right]
\leq \frac{C}{n},
\end{align*}
for some $C>0$.
If $i\neq j$, since $\nu_{k+1}^i\nu_{k+1}^j=0$, we obtain
\begin{equation}\label{E[X k+1 with i and j]}
\EE\left[X_{k+1}^{(i)}X_{k+1}^{(j)}|\Ff_k\right]
= -\quad\sqrt{\frac{q_{k+1}^{(i)}q_{k+1}^{(j)}}{p_{k+1}^{(i)}p_{k+1}^{(j)}}}
= - b_{ij}\left[1 + \Oo\left(\frac{1}{\sqrt{n}}\right)\right],
\end{equation}
and therefore 
$$
\EE\left[\left[W_n^i,W_n^j\right]_t\right]
= -\frac{b_{ij}\lfloor nt \rfloor}{n}\left[1+\Oo\left(\frac{1}{\sqrt{n}}\right)\right].
$$
To prove that $[W_n^i, W_n^j]_t$ converges in $L^2$ to $-b_{ij}t$ for all $t\in[0,1]$, 
it suffices to show that
$\lim_{n\uparrow\infty} \VV\left[\left[W_n^i,W_n^j\right]_t\right] = 0$. 
Indeed,
\begin{align*} \VV\left[\left[W_n^i,W_n^j\right]_t\right]
& = \frac{1}{n^2}\sum_{k=0}^{\lfloor nt \rfloor-1}\EE\left[\left(X_{k+1}^{(i)}X_{k+1}^{(j)}-\EE\left[X_{k+1}^{(i)}X_{k+1}^{(j)}\right]\right)^2\right]\\
& + \frac{2}{n^2}\sum_{0\leq l<k\leq \lfloor nt \rfloor-1}\EE\left[\left(X_{k+1}^{(i)}X_{k+1}^{(j)}-\EE\left[X_{k+1}^{(i)}X_{k+1}^{(j)}\right]\right)
\left(X_{l+1}^{(i)}X_{l+1}^{(j)}-\EE\left[X_{l+1}^{(i)}X_{l+1}^{(j)}\right]\right)\right].
\end{align*}
For all $k\in \NN$, Cauchy-Lipschitz's inequality and Lemma~\ref{lem:def_bij_with_U_noise_op2}
yield
\begin{align*}
\EE\left[\EE\left[\left(X_{k+1}^{(i)}X_{k+1}^{(j)}-\EE\left[X_{k+1}^{(i)}X_{k+1}^{(j)}\right]\right)^2\right]\right]&\leq 2\EE\left[(X_{k+1}^{(i)})^4\right]
\EE\left[(X_{k+1}^{(j)})^4\right]
+ 2\EE\left[(X_{k+1}^{(i)})^2\right]
\EE\left[(X_{k+1}^{(j)})^2\right]\\
&\leq 4K_iK_j.
\end{align*}
Moreover, by~\eqref{E[X k+1 with i and j]},
$\EE\left[X_{k+1}^{(i)}X_{k+1}^{(j)}|\Ff_k\right]-\EE\left[X_{k+1}^{(i)}X_{k+1}^{(j)}\right]
= \Oo\left(\frac{1}{\sqrt{n}}\right)$,
so that, for all $l\leq k$,
\begin{align*}
& \EE\left[\left(X_{k+1}^{(i)}X_{k+1}^{(j)}-\EE\left[X_{k+1}^{(i)}X_{k+1}^{(j)}\right]\right)
\left(X_{l+1}^{(i)}X_{l+1}^{(j)}-\EE\left[X_{l+1}^{(i)}X_{l+1}^{(j)}\right]\right)\right]\\
& \leq \frac{C}{\sqrt{n}} 
\EE\left[\left(X_{l+1}^{(i)}X_{l+1}^{(j)}-\EE\left[X_{l+1}^{(i)}X_{l+1}^{(j)}\right]\right)\right]
\leq \frac{2C}{\sqrt{n}}K_iK_j.
\end{align*}
Overall,
$\VV\left[[W_n^i,W_n^j]_t\right] = \Oo\left(\frac{1}{\sqrt{n}}\right)$,
showing the second condition of Theorem~\ref{th:whitt}, and the proposition thus follows.
\end{proof}

The following shows the convergence of the expected process to an ODE and its proof,
based on~\eqref{X_k centered with U noise} and Proposition~\ref{increment of rho_eps with U nosie op2}, is omitted as analogous to that of Theorem~\ref{Convergence in E with two Cs}.

\begin{theorem}\label{cv in E with U noise op2}
With the setting above and assuming~\eqref{assu:r00},
the sequence $\left(t \mapsto \EE[\rho^\eps_{\lfloor nt\rfloor}(n)]\right)_{n\geq 1}$ 
converges in $L^\infty([0, 1])$ to the unique solution to
$\frac{\D\phi_t^\eps}{\D t}=\Omega^\eps(\phi_t^\eps)$
starting from
$\phi_0^\eps=\rho_0$.
\end{theorem}

The following is the analogue to Lemma~\ref{lemma : uniqueness of solution with alter U}:

\begin{lemma}
Assume~\eqref{assu:r00}, 
let~$W$ be Brownian motion in~$\RR^{3}$
and $\widetilde{B}\in \Ss_3^+(\RR)$ such that $\widetilde{B}^2=B$, with~$B$ in~\eqref{def of B}.
Define $\Xi(\rho) := (\sum_{a=1}^{3} \Theta_{\gamma_a^{(i)} \Kg_a}(\rho))_{i\in\{1,\dots,3\}}$ with the $\gamma_a^{(i)}\in\CC$ depending on~$\Ag$ and the basis $\{e_i\}_{i\in\{0,\dots,3\}}$.
Then the stochastic differential equation
\begin{equation}\label{eq : SDE with U noise op2}
\rho_t^\eps = \rho_{0} + \int_0^t \Omega^\eps(\rho_s^\eps) \D s +\sqrt{\eps} \int_0^t \Xi(\rho_s^\eps)\widetilde{B} \D W_s, 
\end{equation}  
starting from $\rho_{0} \in \Ss(\Hh_0)$ admits a unique strong solution in~$\Ss(\Hh_0)$.
\end{lemma}

\begin{proof}\label{proof of lemma reg of rho eps}
Let $\Xx\in \Cc^\infty([0,\infty),[0,1])$ such that
$\Xx\big|_{[0,1]}=1$ and $\Xx\big|_{[2,\infty)}=0$.
Let~$\widetilde{\Xi}$ the truncation of~$\Xi$ defined by
$\widetilde{\Xi}(\rho) := \Xx(\lVert \rho \rVert)\Xi(\rho)$
for $\rho\in \Ss(\Hh_{0})$.
Then $\rho^\eps$ satisfies the SDE 
\begin{equation}\label{eq : SDE with truncation of Xi}
\rho_t^\eps = \rho_0 + \int_0^t \Omega^\eps(\rho_s^\eps) \D s +\sqrt{\eps} \int_0^t \widetilde{\Xi}(\rho_s^\eps)\widetilde{B} \D W_s.
\end{equation}
Since~$\Omega^\eps$ and~$\widetilde{\Xi}$ are Lipschitz continuous with at most linear growth, the result follows from~\cite[Theorem~5.2.1]{oksendal2003sde}.
\end{proof}

\begin{theorem}\label{cv in d with U nosie op2}
If~\eqref{assu:r00} holds then the sequence $(\rho^\eps_{\lfloor nt \rfloor}(n))_{t\in[0,1]}$ converges in distribution to the unique strong solution to~\eqref{eq : SDE with U noise op2} as~$n$ to infinity.
\end{theorem}

\begin{proof}
The proof is similar to that of Theorem~\ref{th: cv SDE with alternating U},
using Propositions~\ref{cv of W_n with U noise} and~\ref{prop:KurtzWong}. 
\end{proof}

We now analyse how close~$\rho^\eps$ is to~$\rho^0$ to understand whether the solution issued from the unitary with noise~$\Omega^\eps$ converges to the solution issued from~$\Omega^0$. 
From now on, $\Ag$ is a fixed observable on~$\CC^4$ with the assumptions mentioned above, $\{e_i\}_{i\in\{0,\dots,3\}}$ a fixed orthonormal basis on $\CC^4$,
$\rho_0\in\Ss(\Hh_{0})$ and  $\rho_t^0=\E^t\rho_0$ for all $t\in[0,1]$.

\begin{lemma}\label{lemma : regularity of rho eps}
    Let $\eps\in[0,1]$, under the same assumptions of Theorem~\eqref{cv in d with U nosie op2}, let $\rho^\eps$ the solution of the SDE~\eqref{eq : SDE with U noise op2}. Then $\rho^\eps\in \Cc([0,1],\Ss(\Hh_{0}))$ almost surely.
\end{lemma}

\begin{proposition}\label{prop : ctrl of rho_t^eps_rho_t^0}
Let $\rho^\eps$ be the unique solution with values in $\DD([0,1],\Ss(\Hh_{0}))$ of~\eqref{eq : SDE with U noise op2}.
Then
$\sup_{0\leq t\leq 1}
\EE\left[ \left\| \rho_t^\eps -\E^t\rho_0\right\|\right]\leq C\eps$,
for some constant~$C$ depending only on~$\Xi$ and~$B$.
\end{proposition}

\begin{proof}
For any $\rho_1,\rho_2 \in \Ss(\Hh_0)$, 
we have 
$\lVert\Omega^\eps(\rho_1)-\Omega^0(\rho_1)\lVert
 \leq 2\left\| \half \Ig-\rho_1\right\|\leq 2\eps$
and
$\lVert\Omega^\eps(\rho_2)-\Omega^\eps(\rho_1)\rVert \leq (1-\eps)\lVert\rho_2-\rho_1\rVert$.
Then, for any $t\in [0,1]$,
$$
\delta_t
:= \lVert\rho_t^\eps-\rho_t^0\rVert \leq \int_0^t\left\|\Omega^\eps(\rho_s^\eps) - \Omega^0(\rho_s^0)\right\| \D s+ \sqrt{\eps}\left\| \int_0^t \Xi(\rho_s^\eps)\widetilde{B}\D W_s\right\|,
$$
and therefore
\begin{align*}
\delta_t^2&\leq 4\int_0^t \left\|\Omega^\eps(\rho_s^\eps)-\Omega^\eps(\rho_s^0)\right\|^2 + 4\int_0^t\left\|\Omega^\eps(\rho_s^0)-\Omega^0(\rho_s^0)\right\|^2\D s
+2\eps\left\|\int_0^t \Xi(\rho_s^\eps)\widetilde{B}\D W_s\right\|^2\\
&\leq 4(1-\eps)^2\int_0^t \delta_s^2 \D s +4\eps^2 t+2\eps \left\|\int_0^t \Xi(\rho_s^\eps)\widetilde{B} \D W_s \right\|^2.
\end{align*}
Since $\sup_{\rho\in\Ss}\lVert\Xi(\rho)\widetilde{B}\rVert$ is finite and independent of~$t$, then the proof follows as
\begin{align*}
\EE[\delta_t^2]
 & \leq4(1-\eps)^2\int_0^t 
 \EE[ \delta_s^2]\D s + 2\left(2\eps^2+\eps C \right)t\\
 & \leq 4(1-\eps)^2\int_0^t \EE[\delta_s^2]\D s + C\eps t\\
  & \leq C\eps\int_0^ts\E^{-(1-\eps)^2s}\D s,
\end{align*}
by Gr\"onwall's inequality~\cite{clark1987gronwall}.
\end{proof}

\begin{theorem}\label{th : cv in d of Z eps}
Let $\alpha\in[0,\half)$,
$\eps\in(0,1]$ and $\rho^\eps$ the unique solution in~$\DD([0,1],\Ss(\Hh_{0}))$ of~\eqref{eq : SDE with U noise op2}.
The sequence $(\frac{\rho^\eps_t-\E^t\rho_0}{\eps^\alpha})_{\eps>0}$
converges in distribution to zero as $\eps$ tends to zero.
\end{theorem}

\begin{proof}
Let $\eps\in(0,1]$. 
The process $(Z_t^\eps:=\frac{\rho^\eps_t-\E^t\rho_0}{\eps^\alpha})_{t\geq 0}$ satisfies the SDE 
$$
Z_t^\eps = (1-\eps)\int_0^t Z_s^\eps \D s +\eps^{1-\alpha}\int_0^t \left(\half \Ig -\rho_s^0\right) \D s+\eps^{\frac{1}{2}-\alpha}\int_0^t \Xi(\rho_s^\eps)\widetilde{B}d\D W_u
= \int_0^t Z_s^\eps \D s + R^\eps(t),
$$
with $R^\eps(t):=\eps^{\frac{1}{2}-\alpha}\int_0^t \Xi(\rho^\eps_s)\widetilde{B}\D W_s+\eps^{1-\alpha}\int_0^t (\frac{1}{2}\Ig-\rho_s^\eps) \D s$.
We first show that~$R^\eps$ converges to zero in distribution as~$\eps$ tends to zero. 
Since~$\Xi$ is continuous in $\Ss(\Hh_{0})$, there exists $\Ug_\Xi:=\sup_{\rho\in \Ss}\lVert \Xi(\rho)\rVert <\infty$. 
Recall that $(\rho_t^\eps)_t$ is in $\DD([0,1],\Ss(\Hh_{0}))$.
Then the Burholder-Davis-Gundy inequality~\cite[Theorem~3.28]{karatzas2012brownian} yields
$$
\EE\left[\sup_{0\leq t \leq 1}\left\|\int_0^t \Xi(\rho^\eps_s)\widetilde{B}\D W_s\right\|^2\right]
\leq 4 \EE\left[\int_0^1 \rVert \Xi(\rho_s^\eps) \widetilde{B}\rVert^2 \D s\right]
$$
and 
\begin{align*}
& \EE\left[\sup_{0\leq t \leq 1} \lVert R^\eps(t)\rVert^2\right]\\
& \leq 2 \eps^{2(1-\alpha)}
\EE\left[\sup_{0\leq t \leq 1}\left\|\int_0^t \Xi(\rho^\eps_s)\widetilde{B}\D W_s\right\|^2 \right]
+ 2\eps^{1-2\alpha}
\EE\left[\sup_{0\leq t\leq 1}\left\|\int_0^t \left(\half \Ig -\rho^\eps\right) \D s\right\|^2\right]\\
& \leq 2\eps^{2(1-\alpha)}\Ug_\Xi^2\lVert \widetilde{B}\rVert^2+C\eps^{1-2\alpha}
= \Oo\left(\eps^{1-2\alpha}\right).
\end{align*}
For $\eps\in(0,1]$, let
$Y^\eps(t)=Y(t)=(t, \widetilde{B}W(t))^{\top}\in \RR^8$ and $(\eps_n)\in(0,1]^\NN$ converging to~$0$ as~$n$ tends to infinity. 
For all $t\in[0,1]$, since $\EE[[\widetilde{B}W,\widetilde{B}W]_t]=tB$, 
then $\{Y_{\eps_n}\}_n$ is \textit{good} as in~\eqref{def of goodness}.
Applying Proposition~\ref{prop:KurtzWong} with the couple $(R_{\eps_n}, Y_{\eps_n})$, 
the sequence $(Z^{\eps_n})_n$ converges in distribution to
$Z=\int_0^{\cdot} Z_s \D s$ starting from $Z_0=0$, namely to $Z=0$.
\end{proof}

We can actually improve this weak convergence with a large deviations principle for the sequence $(\rho^\eps)_{\eps>0}$.
The following theorem is a direct application of~\cite[Theorem~4.4]{dosreis2018freidlin}.

\begin{theorem}\label{th :LDP for rho eps}
The sequence $\{\rho^\eps\}_{\eps\in(0,1]}$ in~\eqref{eq : SDE with truncation of Xi} satisfies an LDP in $(\Cc([0,1]),\|\cdot\|_\infty)$ with good rate function 
$$
I_{\rho_0}(\varphi) = \half\inf_{g\in\Lambda(\varphi,\Xi)}\int_0^1 \lVert g(t) \rVert^2 \D t,
\qquad\text{for all }\varphi\in\Cc([0,1]),
$$
where $
\Lambda(\varphi,\Xi) := 
\left\{g\in H^3: \varphi(t)=\rho_0+\int_0^t \varphi(s)\D s+\int_0^t \Xi(\varphi(s))g'(s)\D s
\text{ on } [0,1]\right\}$ and $
H := \left\{ h : [0,1] \to \RR: \ h(0) = 0,\ h' \in L^2([0,1])\right\}$ the Cameron-Martin space of Brownian motion.
\end{theorem}

\section{Incorporating memory}\label{sec:memory}
Until now, the evolution of the quantum system was described by discrete or continuous dynamics in which the next state depended only on the current one. 
In particular, the system's trajectory did not retain any explicit memory of past states. This may be too restrictive for certain physical situations where the environment retains partial information from previous interactions and feeds it back into the dynamics. 
We now develop a non-Markovian version in order to keep track of the process' past trajectory.
Let $W$ be a $d'$-dimensional standard Brownian motion and $b : \RR^d \to \RR^d$, $\sigma : \RR^d \to \RR^{d \times d'}$ measurable functions.  
Let $K_b, K_\sigma \in L^2([0,1]; \RR)$ be convolution kernels and $x_0 \in \Cc([0,1]; \RR^d)$.
We consider the Volterra stochastic differential equation
\begin{equation}\label{eq:VolterraSDE}
X_t = x_0(t) + \int_0^t K_b(t - s) b(X_s) \D s + \int_0^t K_\sigma(t - s) \sigma(X_s)\cdot \D W_s.
\end{equation}

We first revisit the discrete deterministic repeated interaction model with memory proposed in~\cite{ciccarello2021collision} (see also~\cite{pleasance2025non} for similar constructions)
and provide a rigorous proof of its convergence to a continuous-time limit of the form~\eqref{eq:VolterraSDE}.
We then propose a stochastic extension of the standard quantum trajectory model by modifying the discrete dynamics to incorporate memory effects directly.
Inspired by the construction in~\cite{ciccarello2021collision}, we interleave interactions and measurements in a non-Markovian way and prove its convergence to a convolutional Volterra SDE.

\subsection{A Volterra-type deterministic evolution from repeated interactions}\label{section 3a}

We start with a rigorous derivation of a Volterra integro-differential equation as the limit of a discrete repeated interaction model.
Our goal is to describe the effective dynamics of a small quantum system~$\Hh_0$ interacting sequentially with an infinite chain of identical environment units~$\Hh$.
The model is based on the construction introduced in~\cite{ciccarello2021collision}, where memory is introduced via random swaps between consecutive environments
and where a continuous-time limit was heuristically stated.
We justify the latter rigorously in a specific asymptotic regime and explicitly characterise it, which turns out to differ slightly from the one proposed in~\cite{ciccarello2021collision}.

For any $k \in \NN$, 
define the swap operator $\widetilde{\Sg}_{k+1,k}$ on the total Hilbert space~$\Hht$ by
$$
\widetilde{\Sg}_{k+1,k} \left( \phi_0 \otimes \bigotimes_{j=1}^\infty \phi_j \right) 
:= \phi_0 \otimes \left(\bigotimes_{j=1}^{k-1} \phi_j\right) \otimes \phi_{k+1} \otimes \phi_k \otimes \left(\bigotimes_{j=k+2}^\infty \phi_j\right),
$$
for $\phi_0\in\Ss(\Hh_0)$ and $\phi_j\in\Ss(\Hh)$ for $j\geq 1$. 
Given $p \in [0,1]$, we define a stochastic swap channel $\Sg_{k+1,k}$ acting on $\Ss(\Hht)$ by
$$
\Sg_{k+1,k}[\sigma]
:= (1 - p)\sigma + p \widetilde{\Sg}_{k+1,k} \sigma\widetilde{\Sg}_{k+1,k}.
$$
Let~$\Hg$ be a Hamiltonian on~$\Hht = \Hh_0\otimes\Hh$ and $\tau > 0$ a fixed interaction time. 
For each $k \in \NN^*$, we define the unitary operator~$\Ug_k$ on~$\Hht$ that acts as
$\Ug := \E^{-\I \tau \Hg}$
on $\Hh_0 \otimes \Hh_k$ and as the identity on all other sub-systems $\{\Hh_i: i\geq 1, i\ne k\}$.
We fix an initial system state $\rho_0\in\Ss(\Hh_0)$, an environment state $\beta\in\Ss(\Hh)$ 
and define the global initial state as
$\sigma_0 := \rho_0 \otimes \bigotimes_{j\geq 1} \beta$.
We then construct the sequence $(\sigma_k)_{k \in \NN}$ of global states via the recursion
$$
\sigma_{k+1} := 
\Ug_{k+1} \ \Sg_{k+1,k}[\sigma_k]  \ \Ug_{k+1}^\dagger, 
\qquad\text{for all }k\geq 0.
$$

The following result, easy to prove, can be found in~\cite[Equations~(207)-(208)]{ciccarello2021collision} and provides a general formula for the evolution of the system:
\begin{lemma}
For all $k \in \NN$, the global state satisfies
$$
\sigma_k = (1 - p) \sum_{j=1}^{k-1} p^{j-1} \Ug_k^j \sigma_{k-j} (\Ug_k^\dagger)^j + p^{k-1} \Ug_k^k \sigma_0 (\Ug_k^\dagger)^k.
$$
Moreover, for each $k \in \NN$, the reduced state $\rho_{k} := \EE_0[\sigma_k]$ belongs to~$\Ss(\Hh_0)$ and satisfies
\begin{equation}\label{eq : relation of rho k memory deter}
\rho_{k} = (1 - p) \sum_{j=1}^{k-1} p^{j-1} 
\EE_0\Big[ \Ug^j (\rho_{k-j} \otimes \beta)(\Ug^\dagger)^j \Big] + p^{k-1} \EE_0 \Big[ \Ug^k (\rho_0 \otimes \beta)(\Ug^\dagger)^k \Big].
\end{equation}
\end{lemma}

We now focus on the asymptotic regime where the system and environment are both qubits, that is $\Hh_0 = \Hh = \CC^2$. 
We consider the interaction time $\tau=\frac{1}{n}$ and a memory parameter $p = p(n) = \E^{-\Gamma / n}$ for some fixed constant $\Gamma > 0$. 
We denote by $(\rho_{k}(n))_k$ the corresponding sequence of reduced states on~$\Hh_0$.
We now prove that the rescaled process $\rho_{\lfloor nt \rfloor}(n)$ converges uniformly to the solution of an integro-differential equation of Volterra type. This gives a mathematically precise version of the memory equation heuristically introduced in~\cite[Equation~(209)]{ciccarello2021collision}, although the form of the limiting equation is slightly different due to our specific scaling.

\begin{theorem}\label{th : cv deter to ODE}
Let $\rho_0\in\Ss(\Hh_0)$, $\beta\in\Ss(\Hh)$,
$\eps_{t}[\rho_{0}] := \EE_0\left[\E^{-\I t \Hg}(\rho_{0} \otimes \beta)\E^{\I t \Hg}\right]$.
The sequence $(\rho_{\lfloor n\cdot \rfloor}(n))_n$ converges in $L^\infty([0,1])$ to the deterministic path $\phi : [0,1] \to \Ss(\Hh_0)$ solution to
\begin{equation}\label{eq limit ODE deterministic memory}
\phi_t = \Gamma \int_0^t \E^{-\Gamma(t-s)} \eps_{t}[\phi_s] \D s + \E^{-\Gamma t} \eps_{t}[\rho_0],
\qquad\text{for all }t \in [0,1].
\end{equation}
\end{theorem}

Since conjugation by a unitary preserves norms, 
then~$\eps_t$ is a bounded operator on $\Ss(\Hh_0)$, 
namely there exists $C > 0$ such that
$\|\E^{-\Gamma t} \eps_{t}\| \leq C$
for all $t \in [0,1]$.
Also, since $t \mapsto \E^{-\Gamma t}\eps_{t}$ is smooth, 
there exists $C > 0$ such that
$\|\E^{-\Gamma s} \eps_{s} - \E^{-\Gamma r} \eps_{r}\| \leq C |s - r|$
for all $s, r \in [0,1]$.
The function~$\phi$  belong to~$\Cc^1$, 
therefore bounded and Lipschitz.
For $t \in [0,1]$, we now show that
$$
\lim_{n\uparrow\infty}
\frac{1}{n} \sum_{j=1}^{\lfloor nt \rfloor -1} \exp\left\{-\Gamma \frac{\lfloor nt \rfloor - j}{n}\right\} \eps_{\frac{\lfloor nt \rfloor - j}{n}}
[\phi_{\frac{j}{n}}]
= \int_0^t \E^{-\Gamma(t-s)} \eps_{t-s}[\phi_s] \D s
$$
in $L^\infty([0,1])$, with an error of order $\oon$. Indeed,
\begin{align*}
&\left\| \frac{1}{n} \sum_{j=0}^{\lfloor nt \rfloor -1} \E^{-\Gamma(t - \frac{j}{n})} 
\eps_{t - \frac{j}{n}}[\phi_{\frac{j}{n}}] - \int_0^{\frac{\lfloor nt \rfloor}{n}} \E^{-\Gamma(t-s)} \eps_{t-s}[\phi_s] \D s \right\| \\
&\leq \sum_{j=0}^{\lfloor nt \rfloor -1} \int_{\frac{j}{n}}^{\frac{j+1}{n}} \left\| \E^{-\Gamma(t - \frac{j}{n})} \eps_{t - \frac{j}{n}}[\phi_{\frac{j}{n}}] - \E^{-\Gamma(t - s)} \eps(t - s)[\phi_s] \right\| \D s \\
&\leq C \sum_{j=0}^{\lfloor nt \rfloor -1} \int_{\frac{j}{n}}^{\frac{j+1}{n}} \left( \|\phi_{\frac{j}{n}} - \phi_s\| + \left|s - \frac{j}{n}\right| \right) \D s \\
& \leq C \sum_{j=0}^{\lfloor nt \rfloor -1} \int_{\frac{j}{n}}^{\frac{j+1}{n}} \left(s - \frac{j}{n}\right) \D s.
\end{align*}
Each term is of order $\frac{1}{n^2}$ and there are at most~$n$ terms, so the total error is $\Oo(1/n)$.
We also estimate the difference between this sum and the one from in the discrete dynamics:
$$
\left\| \frac{1}{n} \sum_{j=0}^{\lfloor nt \rfloor -1} \left( \E^{-\Gamma(t - \frac{j}{n})} \eps_{t - \frac{j}{n}}[\phi_{\frac{j}{n}}] - \E^{-\Gamma \frac{\lfloor nt \rfloor - j}{n}} \eps_{\frac{\lfloor nt \rfloor - j}{n}}[\phi_{\frac{j}{n}}] \right) \right\|
\leq \frac{1}{n} \sum_{j=0}^{\lfloor nt \rfloor -1} C \left| t - \frac{\lfloor nt \rfloor}{n} \right| \leq \frac{C}{n}.
$$
We can therefore write, for any $s\in [0,1]$,
$$
\phi_s = \frac{\Gamma}{n} \sum_{j=1}^{\lfloor ns \rfloor -1} \E^{-\Gamma \frac{\lfloor nt \rfloor - j}{n}} \eps_{\frac{\lfloor nt \rfloor - j}{n}}[\phi_{\frac{j}{n}}] +\E^{-\Gamma\frac{\lfloor nt \lfloor-1}{n}}\eps_{\frac{\lfloor nt \rfloor}{n}}[\rho_0]+ \frac{\chi(s)}{n},
$$
for some bounded function $\chi : [0,1] \to \Ss(\Hh_0)$.
Since $\|\rho^j(n)\| \leq 1$ for all $j$, we also have
$$
\lim_{n \uparrow\infty} \frac{\Gamma}{n} \sum_{j=1}^{\lfloor nt \rfloor -1} \E^{-\Gamma \frac{\lfloor nt \rfloor - j}{n}} \eps_{\frac{\lfloor nt \rfloor - j}{n}}[\rho_(n)]
= \lim_{n \uparrow\infty} \left(1 - \E^{-\frac{\Gamma}{n}}\right)
\sum_{j=1}^{\lfloor nt \rfloor -1} \E^{-\Gamma \frac{\lfloor nt \rfloor - j}{n}} \eps_{\frac{\lfloor nt \rfloor - j}{n}}[\rho_j(n)],
$$
since $1 - \E^{-\frac{\Gamma}{n}} \sim \frac{\Gamma}{n}$.   
Let $v_k(n) := \|\rho_{k}(n) - \phi_{\frac{k}{n}}\|$. From~\eqref{eq : relation of rho k memory deter} then
$v_k(n) \leq \frac{\Gamma C}{n} \sum_{j=1}^{k-1} v_j(n) + \frac{C_1}{n}$
with $C_1 > 0$. 
The discrete Gr\"onwall lemma~\cite{clark1987gronwall} yields
$v_k(n) \leq \frac{C_1}{n} \E^{\frac{\Gamma C k}{n}}$.
In particular, for all $t \in [0,1]$,
$v_{\lfloor nt \rfloor}(n) \leq \frac{C}{n}$,
for $C > 0$ independent of~$t$, which concludes the proof of the first result.
Moreover, for all $t \in [0,1]$, the sequence $(\rho_{\lfloor nt \rfloor}(n))_n$ takes values in the compact set~$\Ss(\Hh_{0})$, thus any pointwise limit also belongs to it.

\subsection{A noisy quantum trajectory model leading to a Volterra-type SDE}\label{section 3b}

We now aim to construct a discrete sequence of quantum trajectories that converges to the solution of a stochastic differential equation (SDE) of Volterra type. The target equation involves exponential convolution kernels in both the drift and the diffusion terms.
We keep the same framework as above, namely working 
under Assumption~\ref{assu:NonDiag}.
Let $p \in [0,1]$ and~$\rho_0\in\Ss(\Hh_0)$, while $\beta\in\Ss(\Hh)$.
We define the initial global state as
$\sigma_0 := \rho_0 \otimes \bigotimes_{j=1}^\infty \beta\in\Ss(\Hht)$
and let $\omb = (\omega_k)_{k \in \NN^*} \in \{0,1\}^{\NN^*}$ represent the outcomes of the successive projective measurements on the environment. We define the sequence of states $(\sigma_k(\omb))_{k \in \NN}$ on the full system as 
$$
\sigma_{k+1}(\omb) :=
p \frac{\Pg^{k+1}_{\omega_{k+1}} \Ug_{k+1}\sigma_k(\omb) \Ug_{k+1}^\dagger \Pg^{k+1}_{\omega_{k+1}}}{\Tr[\Ug_{k+1}\sigma_k(\omb) \Ug_{k+1}^\dagger \Pg^{k+1}_{\omega_{k+1}}]}
+ (1-p) \sigma_0, \qquad \text{starting from }\sigma_0(\omb) = \sigma_0.
$$

This model can be interpreted physically as a noisy version of the deterministic quantum trajectory model discussed in the first section. While the core mechanism — interaction with the environment via unitary evolution followed by projective measurement — remains present, the addition of the convex combination with the fixed state $\sigma_0$ introduces a probabilistic "reset" mechanism governed by the parameter $p$. This models a noisy evolution where, with probability $1-p$, the system returns to the initial state, thus incorporating a stochastic memory loss or decoherence effect into the dynamics.
This construction allows us to rigorously approximate a stochastic Volterra equation with exponential memory kernel using a quantum measurement-based dynamics. The following proposition is an immediate consequence of this construction.

\begin{proposition}
Let $(\sigma_k)_{k\in\NN}$ be the sequence of states defined above. 
If $\sigma_k = \theta_k$, then the random state $\sigma_{k+1}$ takes one of the following values
$$
p\frac{\Pg^{k+1}_i \Ug_{k+1}(\theta_k \otimes \beta) \Ug_{k+1}^\dagger \Pg^{k+1}_i}{\Tr[\Ug_{k+1}(\theta_k \otimes \beta) \Ug_{k+1}^\dagger \Pg^{k+1}_i]}+(1-p)\sigma_0, 
\qquad\text{for } i\in\{ 0, 1\},
$$
with probability
$\Tr[\Ug_{k+1}(\theta_k \otimes \beta) \Ug_{k+1}^\dagger \Pg^{k+1}_i]$.
Furthermore, for any $k\in\NN$,
$\rho_{k}:=\EE_0[\sigma_k] \in \Ss(\Hh_0)$ and if $\rho_{k}=\theta_k$, then~$\rho_{k+1}$ takes one of the following values 
$$
\rho_{k+1}=p\frac{\EE_0\left[(\Ig\otimes \Pg_i)\Ug(\rho_{k}\otimes\beta)\Ug^\dagger(\Ig\otimes \Pg_i)\right]}{\Tr\left[\Ug(\rho_{k}\otimes\beta)\Ug^\dagger(\Ig\otimes \Pg_i)\right]}+(1-p)\rho_0, \qquad \text{for }i\in\{0,1\},
$$
with probability
$\Tr[\Ug(\theta_k \otimes \beta) \Ug^\dagger (\Ig\otimes \Pg_i)]$.
\end{proposition}

By reintroducing the notations $\rho^{j}_{k}$ ($j=0,1$) and the random variables $X_{k+1}$ defined respectively in~\ref{eq:L0 and L1} and~\ref{eq:def_X_k}, 
we obtain the scheme 
$$
\rho_{k} = 
(1-p)\rho_0
+ p\left\{\rhokz + \rhoko - \left(\sqrt{\frac{q_{k+1}}{p_{k+1}}} \rhokz - \sqrt{\frac{p_{k+1}}{q_{k+1}}} \rhoko \right) X_{k+1}\right\}.
$$

From now on, we assume that the parameter $p$ depends on $n$ through $p = p(n) = \E^{-\frac{\Gamma}{n}}$, where $\Gamma > 0$ is a fixed constant. We also assume that the unitary operator $\Ug = \Ug(n)$ satisfies the approximations~\eqref{eq:L00}.
Furthermore, we fix the environment state as $\beta = e_{0}e_{0}^\dagger$ and suppose that the observable~$\Ag$ is non-diagonal in the orthonormal basis $\{e_0,e_1\}$.
Denote by $\rho_{k} = \rho_{k}(n)$ the state of the system at step~$k$.
Since $\rho_{k} \in\Ss(\Hh_0)$ and using the computations in Proposition~\ref{increment of rho_k}, we immediately obtain the following:

\begin{proposition}\label{prop : incr of rho memory}
Under Assumption~\ref{assu:NonDiag}, 
as $n$ tends to infinity,
$$
\rho_{k+1} = \E^{-\frac{\Gamma}{n}}
\left\{\rho_{k}+ \frac{1}{n} \Lg(\rho_{k}) + \oon
+\Big(\Theta_{\Cg_{\gamma}}(\rho_{k}) + o(1)\Big)
\frac{X_{k+1}}{\sqrt{n}}\right\}
+\left(1-\E^{-\frac{\Gamma}{n}}\right)\rho_0,
$$
where $\gamma\in\CC$ is a fixed parameter depending of~$\Ag$ and the basis $\{e_0,e_1\}$. The equality holds uniformly in $k,\omb$.
Hence, as $n$ tends to infinity,
\begin{equation}\label{eq : incr of rho k memory}
\rho_{k}=\rho_0+\sum_{j=0}^{k-1}\E^{-\Gamma\frac{k-j}{n}}\left\{\frac{1}{n} \Lg(\rho_j) + \oon
+\Big(\Theta_{\Cg_{\gamma}}(\rho_j) + o(1)\Big)
\frac{X_{j+1}}{\sqrt{n}}\right\}.
\end{equation}
\end{proposition}

To simplify the notations, we henceforth set the parameter $\gamma = 1$ as in Remark~\ref{rem:gamma1}.
We now present a first result concerning the convergence of the expectations.
Its proof is similar to those of Theorem~\ref{thm:Convergence_in_E} and Theorem~\ref{th : cv deter to ODE} and is therefore omitted.

\begin{theorem}
Under Assumption~\ref{assu:NonDiag}, 
for $\rho_0 \in \Ss(\Hh_0)$,
the sequence $(t \mapsto \EE[\rho_{\lfloor nt \rfloor}(n)])_n$ converges in $L^\infty([0,1])$ to
the solution of $\psi_t = \rho_0 + \int_0^t \E^{-\Gamma(t-s)} \Lg(\psi_s)\D s$.
\end{theorem}

Introduce now the auxiliary processes 
$$
V_n(t) := \frac{\lfloor nt \rfloor}{n}
\qquad\text{and}\qquad
W_n(t) := \sum_{j=0}^{\lfloor nt \rfloor -1} \frac{X_{j+1}}{\sqrt{n}},
$$
as well as $\rho_n(t) := \rho_{\lfloor nt \rfloor}(n)$ and $X_n(t) := \E^{\Gamma \frac{\lfloor nt \rfloor}{n}} \rho_n(t)$. Using~\eqref{eq : incr of rho k memory}, we obtain the relation
\begin{equation} \label{relation of X_n}
    X_n(t) = \E^{\Gamma \frac{\lfloor nt \rfloor}{n}} \rho_0 + \eps_n(t) + \int_0^t \Lg(X_n(s)) \D V_n(s) + \int_0^t \E^{\Gamma s} \Theta_{\Cg}\left(\E^{-\Gamma s} X_n(s)\right) \D W_n(s),
\end{equation}
where $\eps_n(t)$ collects the approximation errors.
The introduction of the process~$X_n$ reveals a structure reminiscent of a classical 
(without memory) SDE. 
Observe that $\| \E^{-\Gamma t} X_n(t) \| \leq 1$
for all $n \in \NN$ and $t \in [0,1]$. 
Fix a smooth cut-off function $\Xx \in \Cc^\infty([0,\infty), [0,1])$ such that
$\Xx\big|_{[0,1]} = 1$ and $\Xx\big|_{[2,\infty)} = 0$,
and define a truncated version of the noise coefficient $\Theta_{\Cg}$, namely $\widetilde{\Theta}_{\Cg} : \Ss(\Hh_{0}) \to \Ss(\Hh_{0})$, by
$\widetilde{\Theta}_{\Cg}(\rho) := \Xx(\| \rho \|) \cdot \Theta_{\Cg}(\rho)$,
so that~$X_n$ also satisfies~\eqref{relation of X_n} with~$\Theta_{\Cg}$ replaced by~$\widetilde{\Theta}_{\Cg}$.
The following lemma, whose proof follows that of~\cite[Theorem~5.2.1]{oksendal2003sde}, ensures uniqueness of a limit to the sequence~$(X_n)_n$.
\begin{lemma}\label{lemma : X well def}
The stochastic differential equation
\begin{equation} \label{eq:limit_SDE_X}
X_t = \E^{\Gamma t} \rho_0 + \int_0^t \Lg(X_s) \D s + \int_0^t \E^{\Gamma s} \widetilde{\Theta}(\E^{-\Gamma s} X_s) \D W_s.
\end{equation}
admits a unique strong solution satisfying $\lVert X_t \rVert \leq \E^{\Gamma t}$
for all $t \in [0,1]$.
\end{lemma}

\begin{theorem}\label{th : cv in d of X n}
Under Assumption~\ref{assu:NonDiag}, 
the sequence~$(X_n)$ in~\eqref{relation of X_n}
converges in distribution to the solution of~\eqref{eq:limit_SDE_X}.
\end{theorem}

\begin{proof}
We apply Proposition~\ref{prop:KurtzWong} to the sequence of augmented processes $(\widetilde{X}_n)_n$, where $\widetilde{X}_n(t) = (t, X_n(t))$ for $t \in [0,1]$. 
Let~$V$ denote the deterministic process $V(t) = t$.
Define the processes $U_n$, $Y_n$ and the function $f: \RR \times \Ss(\Hh_{0}) \to (\RR \times \Ss(\Hh_{0}))^2$ as
$U_n(t) = (t, \E^{\Gamma \frac{\lfloor nt \rfloor}{n}} \rho_0 + \eps_n(t))$, 
$Y_n(t) = ( V_n(t), W_n(t))$
and 
$f(t,X) = ( (0, \Lg(X)),   (0,\E^{\Gamma t} \widetilde{\Theta}(\E^{-\Gamma t} X)))$
so that
$\widetilde{X}_n(t) = U_n(t) + \int_0^t f(\widetilde{X}_n(s^-)) \D Y_n(s)$.
The sequence $(Y_n)_n$ is \textit{good} in the sense of Definition~\ref{def of goodness}, for instance by choosing the stopping times $\tau^\alpha_n \equiv \frac{1+\alpha}{2}$ for all $\alpha > 1$ and $n \in \NN$. This follows similarly to the proof of Theorem~\ref{th: cv SDE with alternating U}. 
According to~\cite[Proposition~4.1]{pellegrini2008existence}, $(\eps_n, V_n, W_n)_n$ converges in distribution to $(0, V, W)$ as~$n$ tends to infinity, where~$W$ is a Brownian motion.
Consequently, we obtain the convergence
$(U_n, Y_n) \Rightarrow (U, Y)$, where
$U(t) = (t, \E^{\Gamma t} \rho_0),\; Y = \left(V, W \right)$.
Finally, since $\|X_n(t)\| \leq \E^{\Gamma t}$ for all $n \in \NN$ and $t \in [0,1]$, and since $\widetilde{\Theta}$ is truncated and continuous on the ball of radius $1$, the function $f$ is bounded and continuous on compact subsets. Therefore, the assumptions of Proposition~\ref{prop:KurtzWong} are satisfied, and the sequence~$(X_n)_n$ converges in distribution to the solution of~\eqref{eq:limit_SDE_X}.
\end{proof}

We now show the corresponding convergence
of the original sequence~$(\rho_n)_n$.

\begin{theorem}\label{thm : cv in d memory}
Under Assumption~\ref{assu:NonDiag}, 
for any $\rho_0 \in \Ss(\Hh_0)$,
there exists a Brownian motion~$W$ such that
$(\rho_n)_n$ 
converges in distribution in $\DD([0,1]$ to the unique strong solution of the Volterra SDE
\begin{equation}\label{limiting SDE rho memory}
\rho_t = \rho_0 + \int_0^t \E^{-\Gamma(t-s)} \Lg(\rho_s) \D s + \int_0^t \E^{-\Gamma(t-s)} \Theta_{\Cg}(\rho_s) \D W_s.
\end{equation}
\end{theorem}

\begin{proof}
Existence and uniqueness of~$\rho$ follow from the fact that $\rho_t = \E^{-\Gamma t} X_t$, where~$X$ is the unique strong solution to~\eqref{eq:limit_SDE_X}.
Theorem~\ref{th : cv in d of X n} establishes that the sequence of processes $\left(t \mapsto \rho_n(t)\E^{\Gamma \frac{\lfloor nt \rfloor}{n}}\right)_n$ converges in distribution to~$X$. We now show that the mapping 
$\Phi : g \mapsto \left(t \mapsto \E^{-\Gamma t} g(t)\right)$
is continuous on $\DD([0,1])$.
Let $(x_n)_n \in (\DD([0,1]))^\NN$ and $x \in \DD([0,1])$ such that $(x_n)$ converges to $x$ in $\DD([0,1])$. Then there exists a sequence of increasing continuous functions $(\theta_n)_n$ from $[0,1]$ to $[0,1]$ such that
$$
\lim_{n \uparrow\infty} \sup_{t \in [0,1]} \|x_n(\theta_n(t)) - x(t)\| = 0
    \qquad \text{and} \qquad \lim_{n \uparrow\infty} \sup_{t \in [0,1]} |\theta_n(t) - t| = 0.
$$
For each $n$ and $t \in [0,1]$, we estimate
\begin{align*}
\left\| \E^{-\Gamma \theta_n(t)} x_n(\theta_n(t)) - \E^{-\Gamma t} x(t) \right\| 
&\leq \left\| x_n(\theta_n(t)) - x(t) \right\| + \|x\|_\infty \left| \E^{-\Gamma \theta_n(t)} - \E^{-\Gamma t} \right| \\
&\leq \left\| x_n(\theta_n(t)) - x(t) \right\| + \Gamma \|x\|_\infty |\theta_n(t) - t|.
\end{align*}
This implies that $\E^{-\Gamma \cdot} x_n(\cdot) \to \E^{-\Gamma \cdot} x(\cdot)$ in $\DD([0,1])$, hence $\Phi$ is continuous.
Applying~$\Phi$ to the sequence $t \mapsto \rho_n(t) \E^{\Gamma \left(\frac{\lfloor nt \rfloor}{n}\right)}$, we deduce that
$\left( t \mapsto \rho_n(t) \E^{\Gamma \left(\frac{\lfloor nt \rfloor}{n} - t \right)} \right)_n \Rightarrow \rho$ 
in $\DD([0,1])$.
Finally, observe that for each $n\geq 1$ and $t \in [0,1]$, using $\|\rho_n(t)\| \leq 1$, we have
    \[
    \left\| \rho_n(t) \left( \E^{\Gamma \left(\frac{\lfloor nt \rfloor}{n} - t \right)} - 1 \right) \right\| \leq \Gamma \left(t - \frac{\lfloor nt \rfloor}{n} \right) \leq \frac{\Gamma}{n}.
    \]
Therefore $\rho_n(t) \E^{\Gamma \left(\frac{\lfloor nt \rfloor}{n} - t \right)}$ and $\rho_n(t)$ are uniformly close, and thus the sequence $(\rho_n)_n$ also converges in distribution to~$\rho$.
For all $t\in[0,1]$, $\rho_t\in\Ss(\Hh_0)$ because $\rho_n(t)\in\Ss(\Hh_0)$.
\end{proof}

\bibliographystyle{siam}
\bibliography{Biblio}

\begin{thebibliography}{10}

\bibitem{attal2006from}
{\sc S.~Attal and Y.~Pautrat}, {\em From repeated to continuous quantum
  interactions}, Annales Henri Poincaré, 7 (2006), pp.~59--104.

\bibitem{belavkin1994nondemolition}
{\sc V.~P. Belavkin}, {\em Nondemolition principle of quantum measurement
  theory}, Foundations of Physics, 24 (1994), pp.~685--714.

\bibitem{billingsley2013convergence}
{\sc P.~Billingsley}, {\em Convergence of Probability Measures}, John Wiley \&
  Sons, 2013.

\bibitem{choi1975completely}
{\sc M.-D. Choi}, {\em Completely positive linear maps on complex matrices},
  Linear Algebra and its Applications, 10 (1975), pp.~285--290.

\bibitem{ciccarello2021collision}
{\sc F.~Ciccarello, S.~Lorenzo, V.~Giovannetti, and G.~M. Palma}, {\em Quantum
  collision models: open system dynamics from repeated interactions}, Physics
  Reports, 954 (2022), pp.~1--70.

\bibitem{clark1987gronwall}
{\sc D.~S. Clark}, {\em Short proof of a discrete {G}ronwall inequality},
  Discrete Applied Mathematics, 16 (1987), pp.~279--281.

\bibitem{dembozeitouni1998LDP}
{\sc A.~Dembo and O.~Zeitouni}, {\em Large Deviations Techniques and
  Applications, 2nd Edition}, vol.~38, Springer, 1998.

\bibitem{dosreis2018freidlin}
{\sc G.~dos Reis, W.~Salkeld, and J.~Tugaut}, {\em {Freidlin-Wentzell LDP in
  path space for McKean-Vlasov equations and the functional iterated logarithm
  law}}, Annals of Applied Probability, 29 (2019), pp.~1487--1540.

\bibitem{gough2004stochastic}
{\sc J.~Gough and A.~Sobolev}, {\em Stochastic {S}chr\"odinger equations as
  limit of discrete filtering}, Open Systems \& Information Dynamics, 11
  (2004), pp.~235--255.

\bibitem{karatzas2012brownian}
{\sc I.~Karatzas and S.~Shreve}, {\em Brownian Motion and Stochastic Calculus,
  2nd Edition}, vol.~113, Springer, 1998.

\bibitem{kraus1983states}
{\sc K.~Kraus, A.~B{\"o}hm, J.~D. Dollard, and W.~Wootters}, {\em States,
  {E}ffects, and {O}perations: {F}undamental {N}otions of {Q}uantum {T}heory},
  vol.~190 of Lecture Notes in Physics, Springer, 1983.

\bibitem{kurtz1991wong}
{\sc T.~G. Kurtz and P.~E. Protter}, {\em Wong–{Z}akai corrections, random
  evolutions, and simulation schemes for {SDEs}}, in Stochastic Analysis,
  Academic Press, Boston, MA, 1991, pp.~331--346.

\bibitem{memin_slominski_stability}
{\sc J.~M\'emin and L.~S{\l}omi\'nski}, {\em Condition {UT} et stabilit\'e en
  loi des solutions d'\'equations diff\'erentielles stochastiques}, S\'eminaire
  de Probabilit\'es, 25 (1991), pp.~162--177.

\bibitem{nielsen_chuang_2000}
{\sc M.~A. Nielsen and I.~L. Chuang}, {\em Quantum Computation and Quantum
  Information}, CUP, 2000.

\bibitem{pellegrini_poisson2007}
{\sc C.~Pellegrini}, {\em Existence, uniqueness and approximation for
  stochastic {S}chr\"odinger equation: the {P}oisson case},  (2007).
\newblock arXiv:0709.3713.

\bibitem{pellegrini2008existence}
\leavevmode\vrule height 2pt depth -1.6pt width 23pt, {\em Existence,
  uniqueness and approximation of a stochastic {S}chr{\"o}dinger equation: the
  diffusive case}, Annals of Probability,  (2008), pp.~2332--2353.

\bibitem{pleasance2025non}
{\sc G.~Pleasance, A.~E. Neira, M.~Merkli, and F.~Petruccione}, {\em
  Non-markovianity in collision models with initial intra-environment
  correlations}, New Journal of Physics, 27 (2025).

\bibitem{whitt2007martingale}
{\sc W.~Whitt}, {\em Proofs of the martingale {FCLT}}, Probability Surveys, 4
  (2007), pp.~268--302.

\bibitem{oksendal2003sde}
{\sc B.~Øksendal}, {\em Stochastic Differential Equations: an Introduction
  with Applications}, Springer, Berlin, 6th~ed., 2003.

\end{thebibliography}

\appendix

\section{Proofs of Section ~\ref{sec:Background}}\label{proofs of section 1}
\subsection{Proof of Proposition~\ref{prop:corresponding_Hamiltonian}}\label{proof of corresp hamiltonian}

The proposition follows from the next two lemmas:

\begin{lemma}\label{lemma : matrix exponential}
For any $X,Y,Z\in \Mm_d(\CC)$, define the  maps $\phi_X$ and $\psi_X$ on $\Mm_d(\CC)$ by 
$$
\phi_X(Y) := \int_0^1 \E^{-sX} Y \E^{sX}\D s 
\qquad\text{and}\qquad \psi_X(Y) := \int_0^1\int_0^s \E^{-sX} Y \E^{(s-r)X} Y \E^{rX}\D r\D s.
$$
Then, as $\eps$ tends to zero,
$$
\exp\Big\{X+\eps Y +\eps^2 Z\Big\} =\E^X\left[\Ig+\eps  \phi_X(Y)+\eps^2\Big(\phi_X(Z)+\psi_X(Y)\Big)\right] + \Oo(\eps^3).
$$
\end{lemma}

\begin{lemma}\label{lem:invertphi_{iD}}
    Let $D$ be a Hermitian matrix and $\lambda_1,\dots,\lambda_d$ its eigenvalues.
    If for all $k\neq l$ $\lambda_k-\lambda_l \notin 2\pi \ZZ\setminus \{0\}$,
    then~$\phi_{\I D}$ defined in Lemma~\ref{lemma : matrix exponential} is invertible.
\end{lemma}

\begin{proof}[Proof of Lemma~\ref{lemma : matrix exponential}]
Since the exponential map has an infinite radius of convergence, then 
$$
\E^{X+\eps Y +\eps^2 Z} = 
 \E^{X} + \eps\sum_{k,j=0}^\infty\frac{X^kYX^j}{(k+j+1)!} + \eps^2\left[\sum_{k,j=0}^\infty\frac{X^kZX^j}{(k+j+1)!} + \sum_{k,j,l=0}^\infty\frac{X^kYX^jYX^l}{(k+j+l+2)!}\right] + \Oo(\eps^3).
$$
Now,
$$
\E^X\phi_X(Y)
 = \sum_{k,j=0}^\infty \int_0^1 \frac{(1-s)^ks^j}{k!j!}\D s X^kYX^j
= \sum_{k,j=0}^\infty\frac{X^kYX^j}{(k+j+1)!}.
$$
Furthermore, $\int_0^1 \int_0^s (1-s)^k(s-r)^jr^l \D r \D s=\frac{k!j!l!}{(k+j+l+2)!}$ for all $k,j,l\in \NN$, and therefore
$$
\E^X\psi_X(Y)
 = \sum_{k,j,l=0}^\infty \int_0^1\int_0^s (1-s)^k(s-r)^jr^l \D r \D s X^kYX^kYX^j
 = \sum_{k,j,l=0}^\infty \frac{X^k Y X^j Y X^l}{(k+j+l+2)!}.
$$
\end{proof}

\begin{proof}[Proof of Lemma~\ref{lem:invertphi_{iD}}]
Since~$D$ is Hermitian, then there exists a unitary~$V$ such that $D=V\Lambda V^\dagger$ with $\Lambda=\text{Diag}(\lambda_1,\dots,\lambda_d)$. 
Given $Y \in \Mm_d(\CC)$ and $\widetilde{Y} := V^\dagger Y V$, we can write
$$
\phi_{\I D}(Y) = V\left(\int_0^1 \E^{-\I s\Lambda}\widetilde{Y}\E^{\I s\Lambda}\D s\right)V^\dagger.
$$ 
For all $k,l\in \{1,\dots,d\}$, we define $\gamma_{kl}$ as
$$\gamma_{kl}:=\int^1_0\E^{\I s(\lambda_l-\lambda_k)}\D s=\begin{cases}
    \displaystyle \frac{\E^{\I(\lambda_l-\lambda_k)}-1}{\I (\lambda_l-\lambda_k)}, &\text{if } \lambda_k\neq \lambda_l,\\
    1, &\text{if } \lambda_k=\lambda_l.
\end{cases}
$$
Therefore $\phi_{\I D}=\operatorname{Conj}(V) \circ H(\gamma) \circ \operatorname{Conj}(V^\dagger)$ with $\operatorname{Conj}(V)$ the conjugation by~$V$ and~$H(\gamma)$ the Hadamard multiplication operator associated with $\gamma$. 
Hence~$\phi_{\I D}$ is invertible if and only if $H(\gamma)$ is. $H(\gamma)$ is invertible if $\gamma_{kl}\neq 0$ for all $k,l$, that is $\lambda_l-\lambda_k \notin2\pi\ZZ\setminus\{0\}$ for $k\neq l$.
\end{proof}

\subsection{Proof of Proposition~\ref{prop:increment_of_rho_k}}\label{proof:prop__increment_of_rho_k}
Let $\Pg_0=(\pf_{ij})$ and $\Pg_1=(\qf_{ij})$ be the orthogonal projectors from Assumption~\ref{assu:NonDiag}.
Using~\eqref{eq:L00}, we can write
\begin{equation*}
\begin{array}{r@{\;}lcr@{\;}l}
\Ug_{00} \rho_{k} \Ug_{00}^\dagger
 & = \displaystyle \rho_{k}-\frac{\I}{n}[\Hg_{0},\rho_{k}]-\frac{1}{2n}\{\Cg \Cg^\dagger,\rho_{k}\} 
 + \oon, 
 & &
 \Ug_{00}\rho_{k}\Ug_{10}^\dagger & = \displaystyle \frac{1}{\sqrt{n}}\rho_{k} \Cg^\dagger + \oon,\\
\Ug_{10}\rho_{k}\Ug_{00} & \displaystyle = \frac{1}{\sqrt{n}}\Cg\rho_{k} + \oon,
& &
 \Ug_{10}\rho_{k}\Ug_{10}^\dagger & = \displaystyle \frac{1}{n}\Cg\rho_{k} \Cg^\dagger + \oon.
\end{array}
\end{equation*}

Recall from~\eqref{eq:discrete_evolution} that
$\displaystyle
\rho_{k+1} = \rhokz + \rhoko - \left(\sqrt{\frac{q_{k+1}}{p_{k+1}}} \rhokz - \sqrt{\frac{p_{k+1}}{q_{k+1}}} \rhoko \right) X_{k+1}$.
Concentrating first on the first two terms, we have, from~\eqref{eq : og mu k+1},
\begin{align*}
\rho^0_k &= \sum_{i,j=0}^1 \pf_{ji}\Ug_{i0}\rho_{k} \Ug_{j0}^\dagger\\
&= \pf_{00}\rho_{k}  + \frac{\pf_{10}\rho_{k}\Cg^\dagger + \pf_{01}\Cg\rho_{k}}{\sqrt{n}}
+ \frac{\pf_{11}\Cg\rho_k \Cg^\dagger-\pf_{00}(\I[\Hg_{0}, \rho_k] + \frac{1}{2} 
\{\Cg \Cg^\dagger, \rho_k \})}{n}
+ \oon.
\end{align*}

Similarly,
\begin{align*}
\rho^1_k &= \sum_{i,j=0}^1 \qf_{ji}\Ug_{i0}\rho_{k} \Ug_{j0}^\dagger\\
&= \qf_{00}\rho_{k}  + \frac{\qf_{10}\rho_{k}\Cg^\dagger + \qf_{01}\Cg\rho_{k}}{\sqrt{n}}
+ \frac{\qf_{11}\Cg\rho_k \Cg^\dagger - \qf_{00}(\I[\Hg_{0}, \rho_k] + \frac{1}{2} 
\{\Cg \Cg^\dagger, \rho_k \})}{n} + \oon
\end{align*}

Since~$\Pg_0$ and~$\Pg_1$ are the two orthogonal projectors, then
$\pf_{00} + \qf_{00} = \pf_{11} + \qf_{11} = 1$ and 
$\pf_{01} + \qf_{0`} = \pf_{01} + \qf_{01} = 0$. 
Therefore
\begin{align*}
\rho^0_k + \rho^1_k
 & = \pf_{00}\rho_{k}  + \frac{\pf_{10}\rho_{k}\Cg^\dagger + \pf_{01}\Cg\rho_{k}}{\sqrt{n}}
+ \frac{\pf_{11}\Cg\rho_k \Cg^\dagger - \pf_{00}(\I[\Hg_{0}, \rho_k] + \frac{1}{2} 
\{\Cg \Cg^\dagger, \rho_k \})}{n}\\
 & \quad + \qf_{00}\rho_{k}  + \frac{\qf_{10}\rho_{k}\Cg^\dagger + \qf_{01}\Cg\rho_{k}}{\sqrt{n}}
+ \frac{\qf_{11}\Cg\rho_k \Cg^\dagger - \qf_{00}(\I[\Hg_{0}, \rho_k] + \frac{1}{2} 
\{\Cg \Cg^\dagger, \rho_k \})}{n} + \oon\\
 & = 
 (\pf_{00} + \qf_{00})\rho_{k}
 + \frac{(\pf_{10}+\qf_{10})\rho_{k}\Cg^\dagger + (\pf_{01}+\qf_{01})\Cg\rho_{k}}{\sqrt{n}}\\
  & \quad + 
 \frac{(\pf_{11}+\qf_{11})\Cg\rho_k \Cg^\dagger - (\pf_{00}+\qf_{00})(\I[\Hg_{0}, \rho_k] + \frac{1}{2} 
\{\Cg \Cg^\dagger, \rho_k \})}{n} + \oon\\
 & = 
 \rho_{k}
 + 
 \frac{\Cg\rho_k \Cg^\dagger - (\I[\Hg_{0}, \rho_k] + \frac{1}{2} 
\{\Cg \Cg^\dagger, \rho_k \})}{n} + \oon 
 =\rho_{k}+\frac{1}{n}\Lg(\rho_k)
 + \oon,
\end{align*}
with~$\Lg(\cdot)$ as in the proposition.
Regarding the terms $p_{k+1} = \Tr[\rhokz]$ and   
$q_{k+1} =\Tr[\rhoko]$,
\begin{equation}\label{eq : pk+1 and qk+1 with U alter}
\left\{
\begin{array}{rl}
p_{k+1} & = \displaystyle  \pf_{00}+\frac{1}{\sqrt{n}}\Tr\left[\pf_{10}\rho_{k}\Cg^\dagger + \pf_{01} \Cg \rho_{k}\right] +\oonsq, \\
q_{k+1} & = \displaystyle \qf_{00}+\frac{1}{\sqrt{n}}\Tr\left[\qf_{10}\rho_{k}\Cg^\dagger + \qf_{01}\Cg \rho_k\right] +\oonsq.
\end{array}
\right.
\end{equation}

Consider now the term
$-\sqrt{\frac{q_{k+1}}{p_{k+1}}} \rhokz + \sqrt{\frac{p_{k+1}}{q_{k+1}}} \rhoko$ in~\eqref{eq:discrete_evolution}.
From the expansions above, 
\begin{align*}
\sqrt{\frac{q_{k+1}}{p_{k+1}}}
 &  = \alpha \left\{1+\frac{1}{2\sqrt{n}}\left(\frac{q}{\qf_{00}} -\frac{p}{\pf_{00}}\right) + \oonsq\right\},\\
 \sqrt{\frac{p_{k+1}}{q_{k+1}}}
 & = \frac{1}{\alpha} \left\{1+\frac{1}{2\sqrt{n}}\left(\frac{p}{\pf_{00}} - \frac{q}{\qf_{00}}\right) + \oonsq\right\},
\end{align*}
with $p:=\Tr[\rho_k(\pf_{10}\Cg^\dagger + \pf_{01}\Cg)]$,
$q:=\Tr[\rho_k(\qf_{10}\Cg^\dagger + \qf_{01}\Cg)]$,
$\alpha:=\sqrt{\frac{\qf_{00}}{\pf_{00}}}$.
Then
\begin{align*}
\zeta := & -\sqrt{\frac{q_{k+1}}{p_{k+1}}} \rhokz + \sqrt{\frac{p_{k+1}}{q_{k+1}}} \rhoko\\
= &-\alpha\left[1+\frac{1}{2\sqrt{n}}\left(\frac{q}{\qf_{00}} - \frac{p}{\pf_{00}}\right)+\oonsq\right]
\left[\pf_{00}\rho_{k} +\frac{1}{\sqrt{n}}(\pf_{10}\rho_{k}\Cg^\dagger+\pf_{01}\Cg\rho_{k})+\oonsq\right]\\
 & + \frac{1}{\alpha} \left[1+\frac{1}{2\sqrt{n}}\left(\frac{p}{\pf_{00}} - \frac{q}{\qf_{00}}\right) + \oonsq\right]
\left[\qf_{00}\rho_{k} +\frac{1}{\sqrt{n}}(\qf_{10}\rho_{k}\Cg^\dagger+\qf_{01}\Cg\rho_{k})+\oonsq\right]\\
 & = \frac{1}{\sqrt{n}}
\left[\left(\alpha p -\frac{q}{\alpha}\right)\rho_{k} + \left(\frac{\qf_{10}}{\alpha}-\alpha \pf_{10}\right)\rho_{k}\Cg^\dagger+\left(\frac{\qf_{01}}{\alpha} - \alpha \pf_{01}\right)\Cg\rho_{k}\right] +\oonsq.
\end{align*}
Since $\alpha p - \frac{q}{\alpha}=\Tr[\rho_{k}[(\alpha \pf_{10}-\frac{\qf_{10}}{\alpha})\Cg^\dagger +(\alpha \pf_{01}-\frac{\qf_{01}}{\alpha})\Cg]]$, $\alpha\ne 0$ and,
with $\gamma=\frac{\qf_{01}}{\alpha}-\alpha \pf_{01}$,
$$
\zeta = \frac{1}{\sqrt{n}}
\Big\{\rho_{k} \overline{\gamma}\Cg^\dagger +\gamma \Cg \rho_{k} -\Tr\left[\rho_{k}(\overline{\gamma}\Cg^\dagger + \gamma \Cg)\right]\rho_{k}+o(1)\Big\}.
$$


\end{document}